\documentclass{amsart}

\pdfoutput=1

\usepackage{amsmath,amssymb,amsthm,comment,mathrsfs}
\usepackage[usenames]{color}
\usepackage[margin=3cm]{geometry}
\usepackage[hyphens]{url}
\usepackage{hyperref}
\usepackage[usenames]{color}
\usepackage{enumitem}
\usepackage[alphabetic,short-journals,short-publishers]{amsrefs}
\usepackage{tikz,tikz-cd}
\usetikzlibrary{matrix,arrows,decorations.pathmorphing}
\usepackage{cleveref}

\hypersetup{pdfstartview={XYZ null null 1.00}, pdfpagemode=UseNone, colorlinks,breaklinks, linkcolor=blue,urlcolor=blue, anchorcolor=blue,citecolor=blue}

\newcommand{\Q}{\mathbb{Q}}
\newcommand{\Z}{\mathbb{Z}}
\newcommand{\R}{\mathbb{R}}
\newcommand{\C}{\mathbb{C}}
\newcommand{\F}{\mathbb{F}}
\newcommand{\A}{\mathbb{A}}

\newcommand{\Qbar}{\overline{\mathbb{Q}}}

\newcommand{\Fbar}{\overline{\F}}

\DeclareMathOperator{\PSL}{PSL}
\DeclareMathOperator{\GSp}{GSp}
\DeclareMathOperator{\GO}{GO}
\DeclareMathOperator{\GL}{GL}

\DeclareMathOperator{\ad}{ad}
\newcommand{\Hom}{\mathrm{Hom}}
\newcommand{\End}{\mathrm{End}}
\newcommand{\Gal}{\mathrm{Gal}}
\newcommand{\M}{\mathrm{M}}

\newcommand{\univ}{\mathrm{univ}}
\newcommand{\lra}{\longrightarrow}

\newcommand{\Frob}{\mathrm{Frob}}
\newcommand{\Nm}{\mathrm{Nm}}

\newcommand{\height}{\mathrm{ht}}
\newcommand{\calO}{\mathcal{O}}

\newcommand{\calS}{\mathcal{S}}

\newcommand{\rhobar}{\overline{\rho}}
\newcommand{\chibar}{\overline{\chi}}
\newcommand{\mubar}{\overline{\mu}}

\newcommand{\epsilonbar}{\overline{\epsilon}}
\newcommand{\calC}{\mathcal{C}}
\newcommand{\frakm}{\mathfrak{m}}
\newcommand{\frakp}{\mathfrak{p}}
\DeclareMathOperator{\Spec}{Spec}

\DeclareMathOperator{\tr}{tr}
\newcommand{\loc}{\mathrm{loc}}

\newcommand{\cris}{\mathrm{cr}}
\newcommand{\pst}{\mathrm{pst}}
\newcommand{\st}{\mathrm{st}}
\newcommand{\rig}{\mathrm{rig}}

\newcommand{\frakq}{\mathfrak{q}}
\newcommand{\frakr}{\mathfrak{r}}

\newcommand{\Art}{\mathrm{Art}}

\newcommand{\WD}{\mathrm{WD}}
\newcommand{\dR}{\mathrm{dR}}

\newcommand{\ab}{\mathrm{ab}}

\newcommand{\CNL}{\mathrm{CNL}}

\newcommand{\frakz}{\mathfrak{z}}
\newcommand{\frakgl}{\mathfrak{gl}}
\newcommand{\fraksl}{\mathfrak{sl}}

\newcommand{\tor}{\mathrm{tor}}
\newcommand{\free}{\mathrm{free}}
\DeclareMathOperator{\Spf}{Spf}

\newcommand{\red}{\mathrm{red}}

\newcommand{\HT}{\mathrm{HT}}
\newcommand{\ord}{\mathrm{ord}}

\newcommand{\tildeS}{\widetilde{S}}
\newcommand{\rbar}{\overline{r}}
\newcommand{\calG}{\mathcal{G}}
\newcommand{\transp}{\,{}^t\!}
\newcommand{\Fss}{F\mbox{-}\mathrm{ss}}

\newcommand{\Iw}{\mathrm{Iw}}
\DeclareMathOperator{\Aut}{Aut}

\newcommand{\bbT}{\mathbb{T}}
\newcommand{\bbG}{\mathbb{G}}
\newcommand{\pol}{\mathrm{pol}}
\newcommand{\mupol}{\mu\mbox{-}\mathrm{pol}}

\DeclareMathOperator{\Lie}{Lie}
\DeclareMathOperator{\Res}{Res}
\DeclareMathOperator{\rec}{rec}

\providecommand{\abs}[1]{\lvert #1 \rvert}

\newcommand{\bibpath}{/home/patrick/Dropbox/Research}

\newtheorem{mainthm}{Theorem}
\crefname{mainthm}{Theorem}{Theorems}

\newtheorem{thm}[subsubsection]{Theorem}
\crefname{thm}{Theorem}{Theorems}
\newtheorem{lem}[subsubsection]{Lemma}
\crefname{lem}{Lemma}{Lemmas}
\newtheorem{prop}[subsubsection]{Proposition}
\crefname{prop}{Proposition}{Propositions}
\newtheorem{cor}[subsubsection]{Corollary}
\crefname{cor}{Corollary}{Corollaries}

\theoremstyle{definition}
\newtheorem{defn}[subsubsection]{Definition}
\crefname{defn}{Definition}{Definitions}

\theoremstyle{remark}
\newtheorem{rmk}[subsubsection]{Remark}
\crefname{rmk}{Remark}{Remarks}

\theoremstyle{definition}

\setenumerate[0]{label=\arabic*.,ref=\arabic*}

\newlist{ass}{enumerate}{1}
\setlist[ass]{label=(\alph*)}
\crefname{assi}{assumption}{assumptions}

\newlist{casez}{enumerate}{1}
\setlist[casez]{label=(\roman*)}
\crefname{casezi}{case}{cases}

\crefalias{enumi}{part}

\title{On automorphic points in polarized deformation rings}
\author{Patrick B. Allen}

\address{Department of Mathematics, University of Illinois at Urbana--Champaign,
Urbana, IL, USA}
\email{pballen@illinois.edu}

\begin{document}

\thanks{The author was partially supported by an AMS--Simons Travel Grant. Some of this work was completed while the author was a guest that the Max Planck Institute for Mathematics, and he thanks them for their hospitality.}

\begin{abstract}
For a fixed mod $p$ automorphic Galois representation, $p$-adic automorphic Galois representations lifting
it determine points in universal deformation space. 
In the case of modular forms and under some technical conditions, B\"{o}ckle showed that every component
of deformation space contains a smooth modular point, which then implies their Zariski density when coupled with the infinite fern of Gouv{\^e}a--Mazur. 
We generalize B\"{o}ckle's result to the context of polarized Galois representations for CM fields, and 
to two dimensional Galois representations for totally real fields. 
More specifically, 
under assumptions necessary to apply a small $R = \mathbb{T}$ theorem and an assumption on the local mod $p$ representation, we prove that every irreducible component of the universal polarized deformation space contains an automorphic point.
When combined with work of Chenevier, this implies new results on the Zariski density of automorphic points in polarized deformation space in dimension three. 
\end{abstract}

\maketitle

\setcounter{tocdepth}{1}

\tableofcontents

\section*{Introduction}\label{sec:intro}

Inspired by the $p$-adic deformation theory of modular forms, Mazur developed a deformation theory for Galois representations that has played a central role in modern algebraic number theory. 
If the fixed mod $p$ Galois representation arises from a modular eigenform, the $p$-adic deformation theory of its Hecke eigensystem can naturally be viewed as part of the $p$-adic deformation theory of the Galois representation. 
A natural (but vague) question is: are these two deformation theories \emph{the same}?

One way to make this precise is to ask whether or not the universal deformation ring for the mod $p$ Galois representation is isomorphic to an appropriate ``big" $p$-adic Hecke algebra. 
Such a result is know as a ``big $R=\bbT$" theorem, and implies that the theories of Galois representations and automorphic forms are intimately connected even when one leaves the geometric world. 
This is further illustrated by Emerton's strategy for proving the Fontaine--Mazur conjecture, which is to start with a big $R=\bbT$ theorem, and then identify the geometric locus inside deformation space with the classical modular points \cite{EmertonLocGlob}.

%

The theory of pseudo-representations often allows one to deduce the existence of a surjection from a universal deformation ring $R^{\univ}$ to a big $p$-adic Hecke algebra $\bbT$. 
In these situations, a big $R=\bbT$ theorem follows from knowing that the classical automorphic points in $\Spec R$ are Zariski dense (at least up to reduced quotients).
The first such result was proved by Gouv\^{e}a and Mazur  \cite{GouveaMazur} in the context of $\GL_2$ over $\Q$, under the assumption that the universal deformation ring is formally smooth, i.e. a power series over $\calO$, of the expected dimension.  
They constructed a sort-of fractal in the rigid analytic generic fibre of $\Spec R$ that they called the \emph{infinite fern}.
The infinite fern shows that the Zariski closure of the modular points has large dimension, and adding the assumption that $R^{\univ}$ is a power series over $\calO$ of the correct dimension implies the Zariski density.

B\"{o}ckle improved on this \cite{BockleDensity} using a novel idea. 
He used explicit computations of local deformation rings \cites{RamakrishnaFinFlat,BockleDemuskin} together with small $R$ equals $\bbT$ theorems (of the type proved by Taylor and Wiles) to show that every irreducible component of universal deformation space contains a smooth modular point, under certain conditions on the residual representation.
The smoothness of the point allows one to deduce that the component has the correct dimension, then the infinite fern implies the modular points are Zariski dense in that component.

In higher dimensions or for more general number fields, the situation is more subtle (see \cite{CalegariMazur}*{\S1.1}), and the naive generalization of the Zariski density statement in the rigid analytic generic fibre is false in general (see \cite{LoefflerDense} and \cite{CalegariEven2}*{\S5}).
However, if one restricts to \emph{polarized} representations of the Galois group of CM or totally real fields, then the situation is much more hopeful, and a precise conjecture was made by Chenevier \cite{ChenevierFern}*{Conjecture~1.15}.
Chenevier \cite{ChenevierFern} expanded and refined the construction of Gouv\^{e}a and Mazur to three dimensional Galois representations for CM fields $F$ in which $p$ is totally split and to two dimensional Galois representations for totally real fields $F$ of even degree over $\Q$ in which $p$ is totally split. 
Chenevier thus proves his conjecture, which is a higher dimensional analogue of the Gouv\^{e}a--Mazur conjecture, in these situations under the additional assumption that the universal deformation ring is formally smooth of the correct dimension. 

In this paper we give a new and more geometric interpretation of B\"{o}ckle's strategy, which affords strong generalization.
In principal, our methods apply any time the ``numerical coincidence" discussed in \cite{CHT}*{\S1} holds. 
We focus on the case of polarized representations of Galois groups of CM fields (in arbitrary dimension), and two dimensional representations of totally real fields.
In order to state some of our main results, we set up some notation.

Let $F$ be a CM field with maximal totally real subfield $F^+$.
Let $S$ be a finite set of finite places of $F^+$ containing all those that ramify in $F$, let $F_S$ be the maximal extension of $F$ unramified outside of the places above those in $S$, and set 
$G_{F,S} = \Gal(F_S/F)$.
Let $E$ be a finite extension of $\Q_p$ with ring of integers $\calO$ and residue field $\F$.
Assume we are given a continuous absolutely irreducible
	\[ \rhobar : G_{F,S} \lra \GL_n(\F) \]
such that $\rhobar^c = \rhobar^\vee \otimes \overline{\epsilon}^{1-n}$, where $c\in G_{F^+}$ is some choice of complex conjugation, $\rhobar^c$ is the conjugate representation given by $\rhobar^c(\gamma) = \rhobar(c\gamma c)$ for all $\gamma \in G_{F,S}$, $\rhobar^\vee$ is the $\F$-linear dual of $\rhobar$, and $\overline{\epsilon}$ is the mod $p$ cyclotomic character. 
Letting $R^{\univ}$ denote the universal deformation ring for $\rhobar$ on the category of complete Noetherian local $\calO$-algebras with residue field $\F$, there is a quotient $R^{\pol}$ of $R^{\univ}$ that is universal for deformations $\rho$ such that $\rho^c \cong \rho^\vee \otimes \epsilon^{1-n}$ (see \S\ref{sec:polardefring}). 
The Galois representations associated to regular algebraic conjugate self dual cuspidal automorphic representations of $\GL_n(\A_F)$ that lift $\rhobar$ naturally yield $\Qbar_p$-points of $\Spec R^{\pol}$. 

\begin{mainthm}\label{thm:intromain}
Fix $\iota : \Qbar_p \xrightarrow{\sim} \C$. 
Assume $p\nmid 2n$ and that every place above $p$ in $F^+$ splits in $F$. 
Assume there is a regular algebraic conjugate self dual cuspidal automorphic representation $\pi$ of $\GL_n(\A_F)$ such that $\rhobar\otimes\Fbar_p$ is isomorphic to the mod $p$ Galois representation attached to $\pi$ and $\iota$.
Assume further:
\begin{ass}
\item Either 
	\begin{itemize}
	\item $\rho_{\pi,\iota}|_{G_v}$ is potentially diagonalizable for each $v|p$ in $F$, or
	\item $\pi$ is $\iota$-ordinary.
	\end{itemize}
\item $\rhobar(G_{F(\zeta_p)})$ is adequate and $\zeta_p \notin F$.
\item For each $v|p$ in $F$, there is no nonzero $\F[G_v]$-equivariant map 
$\rhobar|_{G_v} \rightarrow \rhobar|_{G_v}(1)$.
\end{ass}
Then for any irreducible component $\calC$ of $\Spec R^{\pol}$, there is a regular algebraic conjugate self dual cuspidal automorphic representation whose associated Galois representation determines a $\Qbar_p$-point of 
$\calC$. 
\end{mainthm}

We refer the reader to \cref{thm:mainPD,thm:mainord} below for slightly more general and refined statements,
to \S\ref{sec:Local} for the definition of a potentially diagonalizable representation, to \S\ref{sec:AutGalRep} for a discussion of regular algebraic polarizable cuspidal automorphic representations and their associated Galois representations, and to \cref{def:GLadequate} for the definition of an adequate subgroup of $\GL_n(\F)$. 

The characteristic zero points on universal polarized deformation rings arising from regular algebraic polarized cuspidal automorphic representation are known to be formally smooth in a great deal of generality, see \cite{MeSmooth}*{Theorem~C} and \cite{BHS}*{Corollaire~4.13}. 
Combining this with our main theorems and Chenevier's infinite fern, we obtain new cases of Chenevier's conjecture (\cref{thm:dim3dense} below):

\begin{mainthm}\label{thm:introdense}
Fix $\iota : \Qbar_p \xrightarrow{\sim} \C$. 
Assume that $n = 3$, and that $p>2$. 
Assume there is a regular algebraic conjugate self dual cuspidal automorphic representation $\pi$ of $\GL_3(\A_F)$ such that $\rhobar$ is isomorphic to the mod $p$ Galois representation attached to $\pi$ and $\iota$.
Assume further:
\begin{ass}
	\item $p$ is totally split in $F$.
	\item For each $v|p$ in $F$, $\pi_v$ is unramified and $\rho_{\pi,\iota}|_{G_v}$ is potentially diagonalizable. 
If $p = 3$, then we further assume that $\pi$ is $\iota$-ordinary.
\item $\rhobar(G_{F(\zeta_p)})$ is adequate and $\zeta_p \notin F$.
\item For each $v|p$ in $F$, there is no nonzero $\F[G_v]$-equivariant map $\rhobar|_{G_v} \rightarrow \rhobar|_{G_v}(1)$.
\end{ass}
Then, letting $\mathfrak{X}$ denote the rigid analytic generic fibre of $\Spf R^{\pol}$, the set of points in $\mathfrak{X}$ induced by Galois representations associated to regular algebraic conjugate self dual cuspidal automorphic representation of level prime to $p$ is Zariski dense. 
\end{mainthm}

We refer the reader to \S\ref{sec:genfib} for an overview of the rigid analytic generic fibre.
Along the way to proving \cref{thm:introdense} we deduce nice ring theoretic properties for $R^{\pol}$ (\cref{thm:CMgeom} below). 
However, using potentially automorphy theorems, we can prove these nice ring theoretic properties in many cases without assuming residual automorphy. 
For example:

\begin{mainthm}\label{mainthm:CMgeom}
Assume that $p>2(n+1)$ and that every $v|p$ in $F^+$ splits in $F$. 
Assume further:
\begin{ass}
\item $\rhobar|_{G_{F(\zeta_p)}}$ is absolutely irreducible.
\item For each $v|p$, if we write the semisimplification of $\rhobar|_{G_v}$ as $(\rhobar|_{G_v})^{\mathrm{ss}} = \oplus \rhobar_i$ with each $\rhobar_i$ irreducible, then $\rho_i \not\cong \rho_j(1)$ for any $i,j$.
\end{ass}
Then $R^{\pol}$ is an $\calO$-flat, reduced, complete intersection ring of dimension $1 + \frac{n(n+1)}{2}[F^+:\Q]$.
\end{mainthm}

We refer the reader to \cref{thm:potCMgeom} for a more general statement (see \cref{rmk:PDlift}).
This can be seen as answering a polarized version of a question of Mazur \cite{MazurDefGalRep}*{\S 1.10} in many cases.

We also prove similar theorems in the context of $\GL_2$ over a totally real field, \cref{thm:mainHilbdet,thm:mainHilb,thm:Hilbgeom,thm:Hilbdense} below, and our results even yield new cases over $\Q$ (see \cref{rmk:overQ}).

Most of the assumptions in the above theorems, and in any of the main theorems below, are used to invoke results in the literature. 
In particular, improvements in (small) $R = \bbT$ theorems or improvements in the infinite fern would immediately yield improvements in \cref{thm:intromain,thm:introdense}, respectively.
The only additional assumption imposed in this paper is the assumption that there are no nonzero $\F[G_v]$-equivariant maps $\rhobar|_{G_v} \rightarrow \rhobar|_{G_v}(1)$.
This is used to guarantee that the universal lifting rings of the local representations $\rhobar|_{G_v}$ are regular for each $v|p$. 
We explain how this is used below.

\subsection*{Strategy}
As in B\"{o}ckle's work, we use a small $R = \bbT$ theorem to show that every irreducible component of the universal deformation ring contains an automorphic point. 
Under appropriate assumptions, including one that implies the universal local lifting ring at $p$ is regular, B\"{o}ckle shows that a suitable locus (i.e. finite flat or ordinary) inside the universal local deformation space is cut out by the ``right number" of equations. 
In order to do this, he uses explicit computations of local deformation rings, due to Ramakrishna \cite{RamakrishnaFinFlat} in the finite flat case, and due to himself \cite{BockleDemuskin} in the ordinary case.
In arbitrary dimensions, the local Fontaine--Laffaille deformation rings are known to be power series in the correct number of variables, so one could proceed as in \cite{BockleDensity}.  
However, explicitly computing ordinary deformation rings in arbitrary dimensions seems intractable (even dimension two is hard). 
More importantly, in higher dimensions there are many types of mod $p$ Galois representations that have neither Fontaine--Laffaille nor ordinary lifts, so one would like a strategy that works for more general local conditions.

The main idea of this paper is that using the \emph{finiteness} of the universal global polarized deformation ring $R^{\pol}$ over the universal local lifting ring $R^{\loc}$ at places dividing $p$ (a principle the author first learned from Frank Calegari's blog Persiflage \cite{PersiflageFinite}, see also \cite{MeFrank}) we can turn our problem into one of intersections in $\Spec R^{\loc}$.
This has the effect of allowing us to weaken a \emph{global} unobstructedness assumption to a \emph{local} unobstructedness assumption. 
More specifically, armed with a small $R = \bbT$ theorem, we can deduce the existence of an automorphic point on any irreducible component $\calC$ of $\Spec R^{\pol}$ if we prove that the intersection of $\calC$ and the generic fibre of our small deformation ring $R$ inside $\Spec R^{\pol}$ is nonempty. 
Using the finiteness of $R^{\pol}$ over $R^{\loc}$, we first turn this into a problem of an intersection in $\Spec R^{\loc}$. 
Since $R^{\loc}$ is regular, we can use intersection theory in regular local rings and the lower bound for $\dim\calC$ arising from Galois cohomology to prove that the intersection of the image of $\calC$ in $\Spec R^{\loc}$ with the appropriate fixed weight $p$-adic Hodge theoretic locus $X^{\loc}$ has dimension $\ge 1$. 
One then uses the finiteness of $R^{\pol}$ over $R^{\loc}$ again to show there is a dimension $1$ point on the intersection of $\calC$ and our small deformation ring $R$. 
But the small $R = \bbT$ theorem implies $R$ is finite over $\Z_p$, hence this point must be in the generic fibre.

Let us make a few remarks. 
Firstly, it is crucial for our method that we know our small deformation ring is finite over $\Z_p$. 
Hence, it is not just the automorphy of $p$-adic Galois representations that we need, but the underlying $R = \bbT$ (or $R^{\red} = \bbT$) theorems.

Secondly, the author thinks it is interesting to compare the above with one heuristic justification for the Fontaine--Mazur conjecture. 
Namely, since the conjectural dimension of $\Spec R^{\pol}[1/p]$ plus the dimension of our fixed weight $p$-adic Hodge theoretic locus $X^{\loc}$ in $\Spec R^{\loc}[1/p]$ equals the dimension of $\Spec R^{\loc}[1/p]$ (in favourable situations), one might imagine they intersect at finitely many points, and maybe even transversely. 
The assumption of residual automorphy guarantees that this intersection is nonempty, and a small $R = \bbT$ theorem implies that this intersection is a finite set of points. 
The main theorems in this article and their method of proof (under the appropriate assumptions) imply that if $\Spec R^{\pol}[1/p]$ intersects $X^{\loc}$, then every irreducible component of $\Spec R^{\pol}[1/p]$ intersects $X^{\loc}$.

Finally, since the argument above uses only the dimension of $X^{\loc}$ and that a small $R=\bbT$ theorem is known for the $p$-adic Hodge theoretic conditions defined by $X^{\loc}$, this allows us flexibility in the choice of $X^{\loc}$, provided it is nonzero and has the correct dimension. 
For example, we can use this to conclude existence of automorphic points on each irreducible component of $\Spec R^{\pol}$ whose local representations at places dividing $p$ have certain pre-prescribed local properties.

It is natural to wonder what happens when $R^{\loc}$ is not regular.
We discuss an example in \S\ref{sec:eg} that illustrates the subtleties in this case.
More specifically, we use an example of Serre to show that the conclusion of our main lemma, \cref{thm:thelemma}, may not hold when $R^{\loc}$ is not regular. 
In particular, one no longer has the flexibility in the choice of $X^{\loc}$ as discussed above, and proving the main theorems of this paper in this case seems to require a better understanding of the irreducible components of the universal deformation rings in question.

\subsection*{Outline} 
In \S\ref{sec:mainlem} we first recall a fact from intersection theory in regular local rings and prove our main lemma. 
We then recall the rigid analytic generic fibre of a affine formal scheme, and gather some facts regarding the relationship between the irreducible components of the generic fibre with those in underlying affine scheme.
In \S\ref{sec:GenDef} we recall some basics in deformation theory, in particular presentations and polarized deformation rings.
We gather results from the literature on local deformation rings in \S\ref{sec:Local}. 
In \S\ref{sec:AutGalRep} we recall the notion of regular algebraic polarized cuspidal automorphic representations and their associated Galois representations.
In \S\ref{sec:CM} we prove our main theorems for polarized Galois representations of CM fields. 
We do this after first recalling the small $R = \bbT$ theorems that we use in the proofs of our main theorems.
Finally, \S\ref{sec:Hilb} treats the case of $\GL_2$ over totally real fields.

\subsubsection*{Acknowledgements}
I would like to thank Frank Calegari for many helpful discussions; the main idea used in this article had its genesis in our joint work \cite{MeFrank}. 
I would like to thank Gebhard B\"{o}ckle for comments on a previous version of this article. 
I would like to thank Toby Gee for questions and correspondence that led to an improvement of the main results, namely replacing ``Fontaine--Laffaille" with ``potentially diagonalizable."
Finally, I would like the thank the referee for their corrections and helpful comments.

\section*{Notation and Conventions}
We fix a prime $p$ and an algebraic closure $\Qbar_p$ of $\Q_p$. 
We let $\calO_{\Qbar_p}$ and $\Fbar_p$ denote the ring of integers and residue field, respectively, of $\Qbar$, and let $\frakm_{\Qbar_p}$ be the maximal ideal of $\calO_{\Qbar_p}$.
Throughout $E$ will denote a finite extension of $\Q_p$ inside $\Qbar_p$. 
We let $\calO$ and $\F$ denote the ring of integers and residue field, respectively, of $E$. 

We denote the maximal ideal of a commutative local ring $A$ by $\mathfrak{m}_A$. 
We let $\CNL_{\calO}$ be the category of complete local commutative Noetherian $\calO$-algebras $A$ such that the structure map $\calO \rightarrow A$ induces an isomorphism $\F \xrightarrow{\sim} A/\frakm_A$, and whose morphisms are local $\calO$-algebra morphisms. 
We will refer to an object, resp. a morphism, in $\CNL_{\calO}$ as a $\CNL_{\calO}$-algebra, resp. a $\CNL_{\calO}$-morphism. 

Let $R$ be a commutative ring equipped with a canonical map $R^\square \rightarrow R$, resp. $R^{\univ} \rightarrow R$, with $R^\square$ a universal lifting ring, resp. $R^{\univ}$ a universal deformation ring (see \S\ref{sec:GenDef}). 
Then for any homomorphism $x : R \rightarrow A$, with $A$ a commutative ring, we denote by $\rho_x$ the pushforward of the universal lift, resp. universal deformation, via $R^\square \rightarrow R \xrightarrow{x} A$, resp. via $R^{\univ} \rightarrow R \xrightarrow{x} A$. 

Throughout $F$ will denote a number field. 
A CM field is always assumed to be imaginary. 
If $F$ is CM, its maximal totally real subfield of will be denoted by $F^+$, and $\delta_{F/F^+}$ will denote the nontrivial $\{\pm 1\}$-valued character of $\Gal(F/F^+)$. 
We fix an algebraic closure $\overline{F}$ of $F$ and set $G_F = \Gal(\overline{F}/F)$. 
We will assume that all finite extensions $L/F$ are taken in $\overline{F}$. 
If $L/F$ is a finite Galois extension and $S$ is a finite set of finite places of $F$, we let $L_S$ denote the maximal extension of $L$ that is unramified outside of any of the places in $L$ above those in $S$ and the Archimedean places.
Throughout $\zeta_p$ will denote a primitive $p$th root of unity in $\overline{F}$.

For a finite place $v$ of $F$, we denote by $F_v$ its completion at $v$, and $\calO_{F_v}$ its ring of integers. 
We fix an algebraic closure $\overline{F}_v$ of $F_v$ and let $G_v = \Gal(\overline{F}_v/F_v)$, and 
denote by $I_v$ the inertia subgroup of $G_v$.
We will assume that all finite extensions of $F_v$ are taken inside of $\overline{F}_v$.
We let $\Art_{F_v} : F_v^\times \hookrightarrow G_v^{\mathrm{ab}}$ be the Artin reciprocity map normalized so that uniformizers are sent to geometric Frobenius elements. 
We denote the adeles of $F$ by $\A_F$, and let $\Art_F : F^\times \backslash \A_F^\times \rightarrow G_F^{\ab}$ be $\Art_F = \prod_{v} \Art_{F_v}$.

Let $\Z_+^n$ be the set of tuples of integers $(\lambda_1,\ldots,\lambda_n)$ such that $\lambda_1 \ge \cdots \ge \lambda_n$. 
For any $v|p$ in $F$, we write $\Hom(F_v,\Qbar_p)$ for the set of continuous field embeddings $F_v \hookrightarrow \Qbar_p$. 
If $\sigma \in \Hom(F_v,\Qbar_p)$, we again write $\sigma$ for the induced embedding $F \hookrightarrow \Qbar_p$. 
Conversely, given a field embedding $\sigma : F \hookrightarrow \Qbar_p$, we again write $\sigma$ for the continuous embedding $F_v \hookrightarrow \Qbar_p$ induced by $\sigma$.
If we are given an isomorphism $\iota : \Qbar_p \xrightarrow{\sim} \Omega$ of fields, and $\sigma: K\hookrightarrow \Qbar_p$ is an embedding of fields, we write $\iota\sigma$ for $\iota\circ\sigma$.
If  $r : G \rightarrow \Aut_{\Qbar_p}(V)$ is a representation of a group $G$ on a $\Qbar_p$-vector space $V$, then we will denote by $\iota r$ the representation of $G$ on the $\Omega$-vector space $V\otimes_{L,\iota}\Omega$. 

If $\chi : F^\times \backslash \A_F^\times \rightarrow \C^\times$ is a continuous character whose restriction to the connected component of $(F\otimes \R)^\times$ is given by $x \mapsto \prod_{\sigma \in \Hom(F,\C)} x_\sigma^{\lambda_\sigma}$ for some integers $\lambda_\sigma$, and $\iota : \Qbar_p \xrightarrow{\sim} \C$ is an isomorphism, we let $\chi_\iota : G_F \rightarrow \Qbar_p^\times$ be the continuous character given by 
	\[ \chi_\iota(\Art_F(x)) = \iota^{-1} \Big( \chi(x) \prod_{\sigma \in \Hom(F,\C)} x_\sigma^{-\lambda_\sigma} \Big) 
	\prod_{\sigma\in \Hom(F,\Qbar_p)} x_\sigma^{\lambda_{\iota\sigma}}. \]

We denote by $\epsilon$ the $p$-adic cyclotomic character. 
We use covariant $p$-adic Hodge theory, and normalize our Hodge--Tate weights so that the Hodge--Tate weight of $\epsilon$ is $-1$. 
Let $v$ be a place above $p$ in $F$, and let $\rho : G_v \rightarrow \GL(V) \cong \GL_n(\Qbar_p)$  a potentially semistable representation on an $n$-dimensional vector space over $\Qbar_p$. 
For $\sigma\in \Hom(F_v,\Qbar_p)$, we will write $\HT_\sigma(\rho)$ for the multiset of $n$ Hodge--Tate weights with respect to $\sigma$. 
Specifically, an integer $i$ appears in $\HT_\sigma(\rho)$ with multiplicity equal to the $L$-dimension of the $i$th graded piece of the $n$-dimensional filtered $\Qbar_p$-vector space $D_{\dR}(\rho)\otimes_{(F_v\otimes_{\Q_p}\Qbar_p)} \Qbar_p$, where $D_{\dR}(\rho) = (B_{\dR}\otimes_{\Q_p} V)^{G_v}$, $B_{\dR}$ is Fontaine's ring of de~Rham periods, and we view $\Qbar_p$ as a $F_v\otimes_{\Q_p} \Qbar_p$-algebra via $\sigma \otimes 1$.  
We say that $\rho$ has \emph{regular weight} if $\HT_\sigma(\rho)$ consists of $n$ distinct integers for each $\sigma\in \Hom(F_v,\Qbar_p)$. 
If this is the case, then there is $\lambda \in (\Z_+^n)^{\Hom(F_v,\Qbar_p)}$ such that $\HT_\sigma(\rho) = \{\lambda_{\sigma,j} + n - j\}_{j = 1,\ldots,n}$ for each $\sigma \in \Hom(F_v,\Qbar_p)$, and we call $\lambda$ the \emph{weight} of $\rho$. 

For a finite place $v$ of $F$, an $n$-dimensional \emph{inertial type} is a representation $\tau : I_v \rightarrow \GL_n(\Qbar_p)$ of $I_v$ with open kernel that extends to a representation of the Weil group of $F_v$. 
We say $\tau$ is \emph{defined over} $E$ if it is the extension of scalars to $\Qbar_p$ of a representation valued in $\GL_n(E)$. 
If $\rho : G_v \rightarrow \GL(V) \cong \GL_n(\Qbar_p)$ is a a potentially semistable representation of $G_v$ on an $n$-dimensional vector space over $\Qbar_p$, we say that
$\rho$ has \emph{inertial type} $\tau$ if the restriction to $I_v$ of the Weil--Deligne representation associated to $\rho$ is isomorphic to $\tau$. 
If $v|p$, this is equivalent to demanding that for every $\gamma \in I_v$, the trace of $\gamma$ acting on 
	\[ D_{\pst}(\rho) := \varinjlim_{K/F_v \text{ finite}} (B_{\st} \otimes_{\Q_p} V)^{G_K} \]
equals $\tr\tau(\gamma)$, where $B_{\st}$ denotes Fontaine's ring of semistable periods. 

\section{Irreducible components}\label{sec:mainlem}

In this section, we first prove our main lemma, \cref{thm:thelemma}, that will allow us to deduce the existence of automorphic point in every irreducible component of a (polarized) universal deformation ring from a small $R = \bbT$ theorem. 

We then recall Berthelot's rigid analytic generic fibre and record a lemma that allows us to deduce Zariski density statements in the generic fibre when there are multiple components.

\subsection{Intersections and the main lemma}\label{sec:intersect}
The following lemma is an easy consequence of the intersection theory in regular local rings.

\begin{lem}\label{thm:intersection}
Let $B$ be a regular local commutative ring and let $R$ be a finite commutative $B$-algebra. For $\frakp\in \Spec R$, $\frakq\in \Spec B$ we have
	\[ \dim R/(\frakp,\frakq R) \ge \dim R/\frakp + \dim B/\frakq - \dim B.\]
\end{lem}

\begin{proof}
Let $\frakp_B$ be the pullback of $\frakp$ to $B$. 
If $\mathfrak{r}\in \Spec B$ is minimal containing $\frakp_B + \frakq$, then \cite{SerreLocAlg}*{Chapter V, Theorem 3} implies 
	\[ \height_B\frakr \le \height_B\frakp_B + \height_B\frakq. \]
Since $B$ is catenary, this implies
	\[ \dim B/\frakr \ge \dim B/\frakp_B + \dim B/\frakq - \dim B.\]
Since $R$ is finite over $B$, this in turn implies
	\[ \dim B/\frakr \ge \dim R/\frakp + \dim B/\frakq - \dim B. \]
Again using that $R$ is finite over $B$, we can find $\frakr_R\in \Spec R$ containing $\frakp$ and lying over $\frakr\in\Spec B$, hence also containing $\frakr R \supseteq \frakq R$. Then
	\[ \dim R/(\frakp,\frakq R) \ge \dim R/\frakr_R = \dim B/\frakr \ge \dim R/\frakp + \dim B/\frakq - \dim B.\qedhere\]
\end{proof}

We easily derive from this the following lemma, which is the linchpin in the proofs of our main theorems.
We will use notation suggestive of our intended application to automorphic points in deformation rings. 
Recall that $E$ is a finite extension of $\Q_p$ with ring of integers $\calO$.

\begin{lem}\label{thm:thelemma}
Let $R^{\loc}$ be a local commutative $\calO$-algebra and let $X^{\loc}\subseteq\Spec R^{\loc}$ be a closed subscheme. 
Let $R$ be a commutative $R^{\loc}$-algebra and let $X = X^{\loc} \times_{\Spec R^{\loc}} \Spec R$. 
Let $\calC$ be an irreducible component of $\Spec R$. 
Assume that
	\begin{ass}
	\item $R^{\loc}$ is a regular local ring;
	\item $X$ is finite over $\calO$;
	\item $\dim \calC  + \dim X^{\loc} - \dim R^{\loc} \ge 1$.
	\end{ass}
Then the intersection of $\calC$ with $X\otimes_{\calO} E$ in $\Spec R$ is nonempty.
\end{lem}	

\begin{proof}
Choose $\frakq \in X^{\loc}$ such that $\dim X^{\loc} = \dim R^{\loc}/\frakq$. 
Then $R/\frakq R$ is finite over $\calO$, which implies that $R/\frakm_{R^{\loc}} R$ is finite over $R^{\loc}/\frakm_{R^{\loc}}$, hence $R$ is finite over $R^{\loc}$. Then \cref{thm:intersection} implies that 
	\[ \dim(\calC\cap \Spec R/\frakq R) \ge \dim \calC + \dim R^{\loc}/\frakq - \dim R^{\loc} \ge 1,\]
so there is $\frakq' \in \calC\cap \Spec R/\frakq R$ with $\dim R/\frakq' = 1$. 
Since $R/\frakq R$ is finite over $\calO$, $p\notin \frakq'$, and $\frakq'$ is in the intersection of $\calC$ and $\Spec (R/\frakq R) [1/p]$.
\end{proof}

\subsection{Generic fibres}\label{sec:genfib}
We recall the rigid analytic generic fibre of Berthelot \cite{BerthelotRigCohomSupp}*{\S0.2}, for which we use \cite{deJongCrysDieu}*{\S7} as a reference. 


Let $X = \Spf R$ be an Noetherian adic affine formal $\calO$-scheme such that $R/I$ is a finite type $\F$-algebra, where $I\subset R$ is the largest ideal defining the topology on $R$.
There is a rigid analytic space $X^{\rig}$, called the \emph{rigid analytic generic fibre} of $\Spf R$, that represents the functor that sends an $E$-affinoid algebra $A$ to the set of continuous $\calO$-algebra morphisms $R \rightarrow A$ (see \cite{deJongCrysDieu}*{\S7.1}).
Moreover, $X \mapsto X^{\rig}$ is functorial, and there is a canonical $\calO$-algebra morphism  $R \rightarrow \Gamma(X^{\rig},\calO_{X^{\rig}})$.
%
%
If $R$ is a $\CNL_{\calO}$-algebra (which is the case of interest for us), then $(\Spf R)^{\rig}$ has the following concrete description: if $\calO[[y_1,\ldots,y_g]]/(f_1,\ldots,f_k)$ is a presentation for $\calO$-flat quotient of $R$, then $(\Spf R)^{\rig}$ is isomorphic to the locus in the open rigid analytic unit $n$-ball over $E$ cut out by the equations $f_1 = \cdots = f_k = 0$. 

The following is \cite{deJongCrysDieu}*{Lemma~7.1.9}.

\begin{prop}\label{thm:genfibprop}
Let $X = \Spf R$ be an affine formal $\calO$-scheme as above.
\begin{enumerate}
\item\label{genfibprop:pts} There is a bijection between the points of $X^{\rig}$ and the set of maximal ideals of $R[1/p]$. 
This bijection is functorial in $R$.
\item\label{genfibprop:locring} Let $x\in X^{\rig}$ correspond to the maximal ideal $\frakm \subset R[1/p]$ under the bijection of \cref{genfibprop:pts}. 
There is a canonical morphism of local rings $R[1/p]_{\frakm} \rightarrow \calO_{X^{\rig},x}$. 
This map is compatible with $R \rightarrow \Gamma(X^{\rig},\calO_{X^{\rig}})$, and induces an isomorphism on completions.
\end{enumerate}
\end{prop}


\begin{lem}\label{thm:speczardense}
Let $R$ be an $\calO$-flat $\CNL_{\calO}$-algebra, and let $X = \Spf R$.
Let $Z$ be a set of maximal ideals in $R[1/p]$, and let $Z^{\rig} \subset X^{\rig}$ be the set of points corresponding to $Z$ under \cref{genfibprop:pts} of \cref{thm:genfibprop}.
If $Z^{\rig}$ is Zariski dense in $X^{\rig}$, then $Z$ is Zariski dense in $\Spec R$.
\end{lem}

\begin{proof}
Take $f\in R$ that vanishes at all $\mathfrak{m} \in Z$. 
Since $R$ is $\calO$-flat, it suffices to prove $f$ is nilpotent in $R[1/p]$. 
Since $R[1/p]$ is Jacobson (see \cite{EGA4.3}*{Corollaire~10.5.8}), it further suffices to prove that $f$ belongs to every maximal ideal of $R[1/p]$. 
The Zariski density of $Z^{\rig}$ implies that the image of $f$ under $R \rightarrow \Gamma(X^{\rig},\calO_{X^{\rig}})$ vanishes at all points in $X^{\rig}$, which implies $f$ belongs to every maximal ideal of $R[1/p]$ by \cref{thm:genfibprop}.
\end{proof}

The converse is not true in general. 
For example, let $\widehat{\bbG}_m = \Spf \calO[[t]]$ is the formal multiplicative group over $\calO$, and let $Z$ be the set the maximal ideals in $\calO[[t]][1/p]$ corresponding to $p$-power roots of unity in $\Qbar_p$. 
Then $Z$ is Zariski dense in $\Spec \calO[[t]]$, but $\widehat{\bbG}_m^{\rig}$ is the open rigid analytic unit ball over $E$ with coordinate $t$, and every point in $Z^{\rig}$ is a zero of the analytic function $\log(1+t)$. 
Loeffler \cite{LoefflerDense} has shown that this observation has interesting consequences for universal deformation rings for one dimensional Galois representations.

To apply the principal theorems in this paper to Chenevier's conjecture, we will need to understand the relationship between irreducible components of universal deformation rings and irreducible components of their rigid analytic generic fibre.  
The following lemma follows from a result of Conrad \cite{ConradIrredRig}*{Theorem~2.3.1}.

\begin{lem}\label{thm:genfibcomps}
Let $R$ be an $\calO$-flat, reduced $\CNL_{\calO}$-algebra, and let $X = \Spf R$.
The map that sends a minimal prime ideal $\frakq$ of $R$ to $(\Spf R/\frakq)^{\rig}$ induces a bijection between the irreducible components of $\Spec R$ and the irreducible components of $X^{\rig}$. 
\end{lem}

\begin{proof}
Let $\{\frakq_i\}_i$ be the minimal prime ideals of $R$. 
Since $(\cdot)^{\rig}$ takes closed immersions to closed immersions (see \cite{deJongCrysDieu}*{Proposition~7.2.4}), $(\Spf R/\frakq_i)^{\rig}$ is a closed analytic subvariety of $X^{\rig}$. 
We wish to show that $\{(\Spf R/\frakq_i)^{\rig}\}_i$ is the set of irreducible components of $X^{\rig}$.

Let $\widetilde{R}$ be the normalization of $R$, and let $\phi : \Spf \widetilde{R} \rightarrow \Spf R$ denote the canonical map.
Since $R$ is excellent, the map that sends a maximal ideal $\tilde{\frakm}$ in $\widetilde{R}$ to the minimal prime ideal in $R$ containing $\tilde{\frakm}\cap R$ induces a bijection between connected components of $\Spf \widetilde{R}$ and $\{\frakq_i\}_i$ (see \cite{EGA4.2}*{Scholie~7.8.3(vii)}). 
Let $\widetilde{X}_i$ denote the connected component of $\Spf \widetilde{R}$ corresponding to the minimal prime $\frakq_i$ of $R$, hence $\phi(\widetilde{X}_i) = \Spf R/\frakq_i$.
By \cite{ConradIrredRig}*{Theorem~2.3.1}, the irreducible components of $X^{\rig}$ are $\{\phi^{\rig}(\widetilde{X}_i)\}_i$, and the functoriality of $(\cdot)^{\rig}$ implies $\phi^{\rig}(\widetilde{X}_i) = (\Spf R/\frakq_i)^{\rig}$. 
\end{proof}

In \S\S\ref{sec:CM} and \ref{sec:Hilb}, we will use the above via the following lemma.

\begin{lem}\label{thm:genfiblem}
Let $R$ be an $\calO$-flat, reduced, equidimensional $\CNL_{\calO}$-algebra, and let $X = \Spf R$.
Let $Z$ be a set of maximal ideals in $R[1/p]$, and let $Z^{\rig} \subset X^{\rig}$ be the set of points corresponding to $Z$ under \cref{thm:genfibprop}.
Assume:
	\begin{ass}
	\item\label{genfiblem:dim} Every irreducible component of the Zariski closure of $Z^{\rig}$ has dimension equal to $\dim R[1/p]$.
	\item\label{genfiblem:comps} For every irreducible component $\calC$ of $\Spec R$, there is $\frakm \in Z\cap \calC$ such that $R[1/p]_{\frakm}$ is regular.
	\end{ass}
Then $Z^{\rig}$ is Zariski dense in $X^{\rig}$.
\end{lem}

\begin{proof}
Since $R[1/p]$ is equidimensional, $X^{\rig}$ is equidimensional of dimension $\dim R[1/p]$ by \cref{thm:genfibprop}. 
By \cref{genfiblem:comps} and \cref{thm:genfibcomps}, for every irreducible component $X^{\rig}_i$ of $X^{\rig}$, there is a point $x \in Z^{\rig}\cap X^{\rig}_i$ that lies on no other irreducible component of $X^{\rig}$. 
The lemma now follows from 
\cref{genfiblem:dim} and \cite{ConradIrredRig}*{Corollary~2.2.7}.
\end{proof}
%

\section{General deformation theory}\label{sec:GenDef}

We recall some generalities in the deformation theory of group representations, and fix some notation that will be used in the rest of this article.

\subsection{Universal and fixed determinant deformation rings}\label{sec:general}

Let $\Delta$ be a profinite group satisfying the $p$-\emph{finiteness condition}: for any open subgroup $H$ of $\Delta$, there are only finitely many continuous homomorphisms $H \rightarrow \F_p$. 
This implies that for any $\F$-vector space $M$ with continuous $\F$-linear action of $\Delta$, the cohomology groups $H^i(\Delta,M)$ are all finite dimensional, as is the group of continuous $1$-cocyles $Z^1(\Delta,M)$.

Fix a continuous homomorphism
	\[ \rhobar : \Delta \lra \GL_n(\F).\]
Let $A$ be a $\CNL_{\calO}$-algebra. 
A \emph{lift} of $\rhobar$ to a $\CNL_{\calO}$-algebra $A$ is a continuous homomorphism
	\[ \rho : \Delta \lra \GL_n(A) \]
such that $\rhobar = \rho \mod {\frakm_A}$. 
A \emph{deformation} of $\rhobar$ to $A$ a $1+\M_n(\frakm_A)$-conjugacy class of lifts. 
We will often abuse notation and denote a deformation by a lift in its conjugacy class.
We let $D^\square$, resp. $D$, denote the set valued functor on $\CNL_{\calO}$ that sends a $\CNL_{\calO}$-algebra $A$ to the set of lifts, resp. deformations, of $\rhobar$ to $A$. 
If we wish to emphasize $\rhobar$, we will write $D_{\rhobar}^\square$ and $D_{\rhobar}$, respectively.
The functor $D^\square$ is representable, and so is $D$ if $\End_{\F[\Delta]}(\rhobar) = \F$ (see \cite{BockleDefTheory}*{Proposition~1.3}).

The representing object for $D^\square$, denoted $R^\square$, is called the \emph{universal lifting ring} for $\rhobar$. 
We denote by $\rho^\square$ the universal lift to $R^\square$.
If $\End_{\F[\Delta]}(\rhobar) = \F$, the object representing $D$, denoted $R^{\univ}$, is called the \emph{universal deformation ring} for $\rhobar$. 
We denote by $\rho^{\univ}$ the universal deformation to $R^{\univ}$.
If we wish to emphasize $\rhobar$, we will write $R_{\rhobar}^\square$ and $R_{\rhobar}^{\univ}$, respectively. 

The following well known lemma (see \cite{MazurDefFermat}*{\S12} and \cite{BLGGT}*{Lemma~1.2.1}) allows us to enlarge our coefficient field $E$, and we will sometimes invoke it without comment.

\begin{lem}\label{thm:coefchange}
Let $E'/E$ be a finite extension with ring of integers $\calO'$ and residue field $\F'$. 
Let $\rhobar' = \rhobar \otimes \F$. 
\begin{enumerate}
\item The universal $\CNL_{\calO'}$-lifting ring $R_{\rhobar'}^\square$ is canonically isomorphic to $R_{\rhobar}^\square \otimes_{\calO} \calO'$.
\item If $\End_{\F[\Delta]}(\rhobar) = \F$, the universal $\CNL_{\calO'}$-deformation ring $R_{\rhobar'}$ is canonically isomorphic to $R_{\rhobar'} \otimes_{\calO} \calO'$.
\end{enumerate}
\end{lem}

This lemma has the following consequence that we will use bellow. 
Let
	\[ \rho : \Delta \lra \GL_n(\calO_{\Qbar_p}) \]
be a continuous representation such that $\rhobar \otimes \Fbar_p = \rho \mod {\frakm_{\Qbar_p}}$.
The compactness of $\Delta$ implies there is a finite extension $E'/E$ inside $L$, with ring of integers $\calO'$, such that $\rho$ takes image in $\GL_n(\calO')$. 
Then \cref{thm:coefchange} implies there is a unique local $\calO$-algebra morphism $x : R_{\rhobar}^\square \rightarrow \calO_{\Qbar_p}$ such that $\rho = \rho_x$. 

Let $\ad$ denote the adjoint action of $\GL_n$ on its Lie algebra $\frakgl_n$. 
Let $\ad(\rhobar)$ and $\ad^0(\rhobar)$, denote $\frakgl_n(\F)$ and its trace zero subspace $\fraksl_n(\F)$, respectively, each equipped with the adjoint action $\ad\circ \rhobar$ of $\Delta$. 

\begin{prop}\label{thm:genunivpres}
\begin{enumerate}
\item There is a presentation
	\[ R^\square \cong \calO[[x_1,\ldots,x_g]]/(f_1,\ldots,f_k)\]
with $g = \dim_\F Z^1(\Delta,\ad(\rhobar))$
and $k \le \dim_{\F} H^2(\Delta,\ad(\rhobar))$. 
In particular, every irreducible component of $\Spec R^\square$ has dimension at least 
	\[ 1+ \dim Z^1(\Delta,\ad(\rhobar)) - \dim_\F H^2(\Delta,\ad(\rhobar)),\]
and $R^\square$ is formally smooth over $\calO$ if $H^2(\Delta,\ad(\rhobar)) = 0$.
\item Assume $\End_{\F[\Delta]}(\rhobar) = \F$. 
There is a presentation
	\[ R^{\univ} \cong \calO[[x_1,\ldots,x_g]]/(f_1,\ldots,f_k)\]
with $g = \dim_\F H^1(\Delta,\ad(\rhobar))$ and $k \le \dim_{\F} H^2(\Delta,\ad(\rhobar))$. 
In particular, every irreducible component of $\Spec R^{\univ}$ has dimension at least 
	\[1+ \dim H^1(\Delta,\ad(\rhobar)) - \dim_\F H^2(\Delta,\ad(\rhobar)),\]
and $R^{\univ}$ is formally smooth over $\calO$ if $H^2(\Delta,\ad(\rhobar)) = 0$.
\end{enumerate}
\end{prop}

\begin{proof}
The second part is \cite{BockleLocGlob}*{Theorem~2.4}. 
The first part is proved in the same way, since the tangent space of $D^\square$ is isomorphic to $Z^1(\Delta,\ad(\rhobar)$ via the map $Z^1(\Delta,\ad(\rhobar)) \rightarrow  D^\square(\F[\varepsilon])$ given by $\kappa \mapsto (1+\varepsilon \kappa)\rhobar$,
where $\F[\varepsilon] = \F[\varepsilon]/(\varepsilon^2)$ is the ring of dual numbers over $\F$.
\end{proof}


Now let $\mu : \Delta \rightarrow \calO^\times$ be a continuous character such that $\det\rhobar = \mu \mod {\frakm_{\calO}}$. 
Define subfunctors $D^{\square,\mu} \subseteq D^\square$ and $D^\mu\subseteq D$, that send a $\CNL_{\calO}$-algebra to the set of lifts and deformations, respectively, $\rho$ to $A$ such that $\det\rho = \mu$. 
These subfunctors are easily seen to be represented by the quotient of $R^\square$, resp. of $R$ (assuming $\End_{\F[\Delta]}(\rhobar) = \F$), by the ideal generated by $\{\det\rho^\square(\delta) - \mu(\delta) \mid \delta \in \Delta\}$, resp. generated by $\{\det\rho^{\univ}(\delta) - \mu(\delta) \mid \delta\in \Delta\}$.
The representing object for $D^{\square,\mu}$, denoted $R^{\square,\mu}$, is called the \emph{universal determinant} $\mu$ \emph{lifting ring} for $\rhobar$. 
If $\End_{\F[\Delta]}(\rhobar) = \F$, the object representing $D^\mu$, denoted $R^{\mu}$, is called the \emph{universal determinant} $\mu$ \emph{deformation ring} for $\rhobar$. 

\begin{prop}\label{thm:gendetpres}
Assume $p\nmid n$. 
\begin{enumerate}
\item There is a presentation
	\[ R^{\square,\mu} \cong \calO[[x_1,\ldots,x_g]]/(f_1,\ldots,f_k)\]
with $g = \dim_\F Z^1(\Delta,\ad^0(\rhobar))$
and $k \le \dim_{\F} H^2(\Delta,\ad^0(\rhobar))$. 
In particular, every irreducible component of $\Spec R^{\square,\mu}$ has dimension at least 
	\[ 1+ \dim Z^1(\Delta,\ad^0(\rhobar)) - \dim_\F H^2(\Delta,\ad^0(\rhobar)),\]
and $R^{\square,\mu}$ is formally smooth over $\calO$ if $H^2(\Delta,\ad^0(\rhobar)) = 0$.
\item Assume $\End_{\F[\Delta]}(\rhobar) = \F$. 
There is a presentation
	\[ R^{\mu} \cong \calO[[x_1,\ldots,x_g]]/(f_1,\ldots,f_k)\]
with $g = \dim_\F H^1(\Delta,\ad^0(\rhobar))$ and $k \le \dim_{\F} H^2(\Delta,\ad^0(\rhobar))$. 
In particular, every irreducible component of $\Spec R^\mu$ has dimension at least 
	\[1+ \dim H^1(\Delta,\ad^0(\rhobar)) - \dim_\F H^2(\Delta,\ad^0(\rhobar)),\]
and $R^\mu$ is formally smooth over $\calO$ if $H^2(\Delta,\ad^0(\rhobar)) = 0$.
\end{enumerate}
\end{prop}

\begin{proof}
This is similar to \cref{thm:genunivpres} above. 
For example, see \cite{KisinModof2}*{Lemma~4.1.1}
(since we have assumed $p\nmid n$, $\ad(\rhobar) = \ad^0(\rhobar)\oplus \F$ is $\Delta$-equivariant, so the groups denoted $H^i(G,\ad^0 V)'$ in \cite{KisinModof2}*{\S 4.1} are just $H^i(\Delta,\ad^0(\rhobar))$ in our case). 
\end{proof}


\subsection{Polarized deformation rings}\label{sec:polardefring}
We now assume that $\Delta$ is an open index two subgroup of a profinite group $\Gamma$, and that there is $c\in \Gamma\smallsetminus \Delta$ of order $2$. 
So $\Gamma$ is the semidirect product of $\Delta$ and $\{1,c\}$. 
For any commutative ring $A$ and homomorphism $\rho : \Delta \rightarrow \GL_n(A)$, we let $\rho^c$ denote the conjugate homomorphism, i.e. $\rho^c(\delta) = \rho(c\delta c)$ for all $\delta\in \Delta$, and let $\rho^\vee$ denote the $A$-linear dual of $\rho$, i.e. $\rho^\vee(\delta) = {}^t \rho(\delta)^{-1}$ for all $\delta\in \Delta$. 

Let $\mu : \Gamma \rightarrow \calO^\times$ be a continuous character, and let $\mubar: \Gamma\rightarrow \F^\times$ be its reduction mod $\frakm_{\calO}$. 
We assume that 
	\[ \rhobar^c \cong \rhobar^\vee \otimes\mubar.\]
We then define a subfunctor $D^{\pol} \subseteq D$ by letting $D^{\pol}(A)$, for a $\CNL_{\calO}$-algebra $A$, to be the subset of deformations $\rho$ of $\rhobar$ to $A$ such that $\rho^c \cong \rho^\vee \otimes \mu$. 

\begin{prop}\label{thm:poldefrep}
Assume $\rhobar$ is absolutely irreducible. 
Then $D^{\pol}$ is representable by a quotient $R^{\pol}$ of $R^{\univ}$.
\end{prop}

\begin{proof}
Let $R^{\pol}$ be the quotient of $R^{\univ}$ by the ideal generated by 
\[ \{\tr\rho^{\univ}(c\delta c) - \tr{}^t \rho^{\univ}(\delta)^{-1}\mu(\delta)
\mid \delta \in \Delta\}.\]
The result now follows from \cite{CarayolAnneauLocal}*{Th\'{e}or\`{e}me~1}.
\end{proof}

If $\rhobar$ is absolutely irreducible, we call the object $R^{\pol}$ representing $D^{\pol}$ the \emph{universal} $\mu$-\emph{polarized deformation ring} of $\rhobar$. 
If we wish to emphasize the role of $\mu$, we will write $R^{\mupol}$ for $R^{\pol}$.

We recall the Clozel--Harris--Taylor group scheme $\mathcal{G}_n$, which is the group scheme over $\Z$ defined as the semidirect product
	\[ (\GL_n \times \GL_1) \rtimes \{1,\jmath\} = \calG_n^0 \rtimes \{1,\jmath\}, \]
where $\jmath (g,a) \jmath = (a \transp g^{-1},a)$, and the homomorphism $\nu : \mathcal{G}_n \rightarrow \GL_1$ given by $\nu(g,a) = a$ and $\nu(\jmath) = -1$.
If $A$ is any commutative ring and $r : \Gamma \rightarrow \calG_n(A)$ is a homomorphism such that $r(\Delta) \subseteq \calG_n^0(A)$, we will write $r|_\Delta$ for the composite of the restriction of $r$ to $\Delta$ with the projection $\calG_n^0(A) \rightarrow \GL_n(A)$. 
In particular, we view $A^n$ as an $A[\Delta]$-module via $r|_{\Delta}$. 
The following is (part of) \cite{CHT}*{Lemma~2.1.1}

\begin{lem}\label{thm:Gdhoms}
Let $A$ be a topological ring. 
There is a natural bijection between the following two sets.
\begin{enumerate}
	\item Continuous homomorphisms $r : \Gamma \rightarrow \calG_n(A)$ inducing an isomorphism $\Gamma/\Delta \xrightarrow{\sim} \calG_n(A)/\calG_n^0(A)$.
	\item Triples $(\rho,\eta,\langle\cdot,\cdot\rangle)$, where $\rho : \Delta \rightarrow \GL_n(A)$ and $\eta : \Gamma \rightarrow A^\times$ are continuous homomorphisms and $\langle \cdot, \cdot \rangle$ is a perfect $A$-linear pairing on $A^n$ satisfying
		\[ \langle \rho(\delta) a, \rho(c\delta c) b \rangle = \eta(\delta)\langle a, b \rangle \quad \text{and}
		\quad \langle a, b \rangle = - \eta(c) \langle b, a \rangle \]
	for all $a,b\in A^n$ and $\delta \in \Delta$.
\end{enumerate}
Under this bijection, $\rho = r|_\Delta$, $\eta = \nu\circ r$, and $\langle a, b \rangle = {}^t a P^{-1} b$ for $r(c) = (P,-\eta(c))\jmath$.
\end{lem}

In particular, our fixed $\rhobar$ extends to a continuous homomorphism
	\[ \rbar : \Gamma \lra \calG_n(\F),\]
such that $\nu\circ\rbar = \mubar$, which we fix. 
For a $\CNL_{\calO}$-algebra $A$, a \emph{lift} of $\rbar$ to $A$ is a continuous homomorphism $r \rightarrow \calG_n(A)$ such that $r \mod \mathfrak{m}_A = \rbar$.
A \emph{deformation} of $\rbar$ to $A$ is a $1+\M_n(\frakm_A)$-conjugacy class of lifts.
By \cref{thm:Gdhoms}, if $r$ is a deformation of $\rbar$ to a $\CNL_{\calO}$-algebra $A$, then $r|_\Delta$ is a deformation of $\rhobar$ to $A$. 
We say a lift or a deformation $r$ of $\rbar$ \emph{has multiplier} $\mu$ if $\nu \circ r = \mu$. 
We let $D_{\rbar}^{\pol}$ be the set valued functor on $\CNL_{\calO}$ that sends a $\CNL_{\calO}$-algebra $A$ to the set of deformations of $\rbar$ to $A$ with multiplier $\mu$.

\begin{prop}\label{thm:polisofunctors}
Assume $p>2$ and $\rhobar$ is absolutely irreducible. 
The map $D_{\rbar}^{\pol} \rightarrow D_{\rhobar}^{\pol}$ given by $r \mapsto r|_\Delta$ is an isomorphism of functors.
\end{prop}

\begin{proof}
This is exactly as in \cite{ChenevierFern}*{Lemma~1.5}. 
As the proof is short, we reproduce it for completeness.
Let $A$ be a $\CNL_{\calO}$-algebra. 
The map $D_{\rbar}^{\pol}(A) \rightarrow D_{\rhobar}^{\pol}(A)$ is surjective by \cref{thm:Gdhoms}. 
Assume that $r_1,r_2$ are two lifts of $\rbar$ to $A$ such that $r_i|_{\Delta}$ are $1+\M_n(\frakm_A)$-conjugate. 
Then replacing $r_2$ with a lift in its deformation class, we can assume $r_1|_\Delta = r_2|_\Delta$. 
Letting $P_i$ be given by $r_i(c) = (P_i,-\mu(c))\jmath$ as in \cref{thm:Gdhoms},  $P_1P_2^{-1}$ commutes with $r_1|_\Delta = r_2|_\Delta$. 
Since $\rhobar$ is absolutely irreducible, $P_1 P_2^{-1} = \beta \in A^\times$; since $r_1$ and $r_2$ both lift $\rbar$, $\beta \in 1+\frakm_A$; and since $p>2$, there is $\alpha \in 1+\frakm_A$ such that $\alpha^2 = \beta$. 
Then $r_1 = \alpha r_2 \alpha^{-1}$, so $r_1$ and $r_2$ define the same deformation.
\end{proof}

In particular, when $p>2$ and $\rhobar$ is absolutely irreducible, we view $R^{\pol}$ as the universal ($\mu$-polarized) deformation ring for $\rbar$.

Since $\frakgl_n = \Lie \GL_n \subset \Lie \mathcal{G}_n$, the adjoint action of $\GL_n$ on $\frakgl_n$ extends to $\calG_n$ by 
	\[ \ad(g,a)(x) = gxg^{-1} \quad \text{and} \quad \ad(\jmath)(x) = -{}^t x.\]
If $A$ is a commutative ring, and 
	\[ r : \Gamma \lra \calG_n(A) \]
is a homomorphism, we write $\ad(r)$ for $\frakgl_n(A)$ with the adjoint action $\ad\circ r$ of $\Gamma$.

\begin{prop}\label{thm:genpolpres}
Assume $p>2$ and $\rhobar$ is absolutely irreducible.
There is a presentation
	\[ R^{\pol} \cong \calO[[x_1,\ldots,x_g]]/(f_1,\ldots,f_k)\]
with $g = \dim_{\F} H^1(\Gamma,\ad(\rbar))$ and $k \le \dim_{\F} H^2(\Gamma,\ad(\rbar))$. 
In particular, every irreducible component of $\Spec R^{\pol}$ has dimension at least 
	\[1+\dim_{\F} H^1(\Gamma,\ad(\rbar)) - \dim_{\F} H^2(\Gamma,\ad(\rbar)),\]
and $R^{\pol}$ is formally smooth over $\calO$ if $H^2(\Gamma,\ad(\rbar)) = 0$.
\end{prop}

\begin{proof}
This is a special case of \cite{CHT}*{Lemma~2.2.11 and Corollary~2.2.12}. 
(In our case, the sets denoted $S$ and $T$ in \cite{CHT}*{Lemma~2.2.11 and Corollary~2.2.12} are both empty. 
So in the notation of \cite{CHT}, $R_{\mathscr{S},T}^{\loc} = \calO$ and $H^i_{\mathscr{S},T}(\Gamma,\ad(\rbar)) =H^i(\Gamma,\ad(\rbar))$).
\end{proof}

\begin{rmk}\label{rmk:GL2pol}
We could also define polarized deformation functor in the case that $\rhobar \cong \rhobar^\vee \otimes\mubar$, which amounts to $\GSp_n$ or $\GO_n$-valued deformation theory with fixed multiplier character $\mu$. 
In particular, when $n=2$, $p>2$, and $\mu$ lifts $\det\rhobar$, the universal $\mu$-polarized deformation functor equals the universal determinant $\mu$ deformation functor.
It is useful to keep this in mind when comparing the results of \S\cref{sec:CM} and \S\cref{sec:Hilb}.
\end{rmk}

\section{Local deformation theory}\label{sec:Local}

In this section we recall some results from the literature on local deformation rings that we will need later.
\subsection{Setup}\label{sec:locsetup}
Let $v$ be a finite place of $F$. 
Fix a continuous representation
	\[ \rhobar_v : G_v \lra \GL_n(\F), \]
and a continuous character $\mu : G_v \rightarrow \calO^\times$ with $\mu \bmod {\frakm_{\calO}} = \det \rhobar_v$.
Denote by $R_v^\square$ the universal lifting ring for $\rhobar_v$, and $\rho_v^\square$ the universal lift.
Let $R_v^{\square,\mu}$ be the universal determinant $\mu$ lifting ring, and $\rho_v^\mu$ the universal $R_v^{\square,\mu}$-valued lift.

We refer the reader to \cite{CHT}*{Definition~2.2.2} for the notion of a \emph{local deformation problem}. 
We will primarily use the following characterization (\cite{CHT}*{Lemma~2.2.3}):
\begin{itemize}
\item Any local deformation problem is representable by a quotient of $R_v^\square$.
\item A quotient $R$ of $R_v^\square$ represents a local deformation problem if and only if it satisfies the following: for any $\CNL_{\calO}$-algebra $A$, lift $\rho$ of $\rhobar_v$ to $A$, and $g \in 1+\M_n(\frakm_A)$, the $\CNL_{\calO}$-algebra morphism $R_v^\square \rightarrow A$ induced by $\rho$ factors through $R$ if and only if the $\CNL_{\calO}$-algebra map $R_v^\square \rightarrow A$ induced by $g\rho g^{-1}$ factors through $R$. 
\end{itemize}

\begin{lem}\label{thm:defquolem}
Let $R$ be an $\calO$-flat reduced quotient of $R_v^\square$.
\begin{enumerate}
\item\label{defquolem:flat} $R$ represents a local deformation problem if and only if for every finite totally ramified extension $E'/E$ with ring of integers $\calO'$, lift $\rho$ of $\rhobar$ to $\calO'$, and $g \in 1+\M_n(\frakm_{\calO'})$, the $\CNL_{\calO}$-algebra morphism $R_v^\square \rightarrow \calO'$ induced by $\rho$ factors through $R$ if and only if the $\CNL_{\calO}$-algebra map $R_v^\square \rightarrow \calO'$ induced by $g\rho g^{-1}$ factors through $R$. 
\item\label{defquolem:irred} $R$ represents a local deformation problem if and only if $R/\frakq$ represents a local deformation problem for each minimal prime $\frakq$ of $R$.
\end{enumerate}
\end{lem}

\begin{proof}
Let $g \in 1+\M_n(\frakm_{R_v^\square})$. 
Then the lift $g\rho_v^\square g^{-1}$ induces a $\CNL_{\calO}$-algebra morphism $\phi_g : R_v^\square \rightarrow R_v^\square$. 
It is easy to see that a quotient $f : R_v^\square \rightarrow R'$ represents a local deformation problem if and only if $f \circ \phi_g = f$ for every $g\in  1+\M_n(\frakm_{R_v^\square})$. 

Using this reformulation, \cref{defquolem:flat} follows from the fact that the maximal ideals in $R_v^\square[1/p]$ are Zariski dense (\cite{EGA4.3}*{Proposition~10.5.3}),  since $R$ is $\calO$-flat and reduced and the residue field for any maximal ideal of $R_v^\square[1/p]$ is a finite totally ramified extension of $E$, (\cite{KW2}*{Proposition~2.2}).
It similarly follows that $R$ represents a local deformation problem if $R/\frakq$ does for each minimal prime $\frakq$ of $R$, and the converse follows from the argument of \cite{BLGGT}*{Lemma~1.2.2}. 
\end{proof}




\subsection{Residual characteristic $\ne p$}\label{sec:notp}
We first assume that $v\nmid p$. 
We recall a quotient of $R_v^\square$ studied by Taylor \cite{TaylorIHES}*{Proposition~3.1} (where it is denoted $R_v^{\loc}/\mathcal{J}_v^{(1,\ldots,1)}$) and \cite{ThorneAdequate}*{Proposition~3.12}, that will appear when recalling results from the literature in \S\ref{sec:CMsmallRT}.


\begin{prop}\label{thm:R1}
Assume $\rhobar_v$ is trivial. 
There is a quotient $R_v^1$ of $R_v^\square$ representing the following local deformation problem: for a lift $\rho$ of $\rhobar$ to a $\CNL_{\calO}$-algebra $A$, the induced $\CNL_{\calO}$-algebra morphism $R_v^\square \rightarrow A$ factors through $R_v^1$ if and only if for every $\gamma\in I_v$, the characteristic polynomial of $\rho(\gamma)$ is $(X-1)^n$. 
The $\CNL_{\calO}$-algebra $R_v^1$ is equidimensional of dimension $n^2+1$ and every generic point of $\Spec R_v^1$ has characteristic $0$.
\end{prop}

The construction of $R_v^1$ can be seen easily. 
Indeed, if $\rhobar_v$ is trivial then $\rho_v^\square|_{I_v}$ must factor through tame inertia. 
Letting $X^n - a_{n-1}X^{n-1} + \cdots + (-1)^n a_0$ denote the characteristic polynomial for a generator of tame inertia, we let $R_v^1$ be the quotient of $R_v^\square$ by the ideal generated by $\{a_j - \binom{n}{j}\}_{j=0,\ldots,n-1}$.

The following seems to be well known, but for lack of a concrete reference we include a proof.

\begin{lem}\label{thm:R1factor}
Let $A$ be a reduced $\CNL_{\calO}$-algebra and let $\rho : G_v \rightarrow \GL_n(A)$ be a lift of $\rhobar_v$. 
There is a finite extension $K/F_v$ such that 
the $\CNL_{\calO}$-algebra morphism $R_{\rhobar_v|_{G_K}}^\square \rightarrow A$ induced by $\rho|_{G_K}$ factors through $R_{\rhobar_v|_{G_K}}^1$. 
\end{lem}

\begin{proof}
Assume first that $A$ in an integral domain. 
There is a finite extension $L/F_v$ such that the image of the wild ramification under $\rho|_{G_L}$ is trivial.
Letting $t$ denote a generator of the tame inertia of $G_L$, $\Phi$ a lift of Frobenius to $G_L$, and $q$ the order of the residue field of $L$, the relation $\Phi^{-1}t \Phi = t^q$ implies that the eigenvalues of $\rho(t)$, in an algebraic closure of the fraction field of $A$, are stable under the taking $q$th powers, hence are roots of unity. 
Passing to another finite extension $K/L$ to trivialize these eigenvalues, the image of inertia under $\rho|_{G_K}$ is unipotent. 

Now assume $A$ is only reduced. 
By the above discussion, for each minimal prime ideal $\frakq$ of $A$, there is a a finite extension $K_{\frakq}/F_v$ such that $\rho(I_{K_{\frakq}})$ is unipotent. 
Since $A$ is reduced, there is an injection
	\[ \GL_n(A) \hookrightarrow \prod_{\frakq} \GL_n(A/\frakq),\]
the product the product being taken over all minimal prime ideals of $A$. 
We can then take $K$ to be the composite of the $K_{\frakq}$. 
be any finite extension of $F_v$ containing each $K_{\frakq}$, and such that $\rhobar_v|_{G_K}$ is trivial.
\end{proof}

\subsection{Residual characteristic $p$}\label{sec:p}
For the remainder of this section we assume that $v|p$. 

\begin{lem}\label{thm:localsmooth}
\begin{enumerate}
\item If $H^0(G_v,\ad(\rhobar_v)(1)) = 0$, then $R_v^\square$ is isomorphic to a power series over $\calO$ in $n^2(1 + [F_v:\Q_p])$ variables.
\item\label{localsmooth:det} If $p\nmid n$ and $H^0(G_v,ad^0(\rhobar_v)(1)) = 0$, then $R_v^{\square,\mu}$ is isomorphic to a power series over $\calO$ in $(n^2 - 1)(1+ [F_v:\Q_p])$ variables.
\end{enumerate}
\end{lem}

\begin{proof}
The trace pairing on $\ad(\rhobar_v)$ is perfect, and induces a perfect pairing on $\ad^0(\rhobar)$ if $p\nmid n$. 
By Tate local duality, $\dim_{\F} H^2(G_v,\ad(\rhobar_v)) = \dim_{\F}H^0(G_v,\ad(\rhobar_v)(1))$, and $\dim_{\F} H^2(G_v,\ad^0(\rhobar)) = \dim_{\F} H^0(G_v,\ad^0(\rhobar)(1))$ if $p\nmid n$. 
The result then follows from part 1 of \cref{thm:genunivpres,thm:gendetpres}, and the local Euler--Poincar\'{e} characteristic formula.
%
\end{proof}

The following fundamental result is due to Kisin \cite{KisinPssDefRing}*{Corollary~2.7.7 and Theorem~3.3.8} 

\begin{thm}\label{thm:crdefring}
Fix $\lambda \in (\Z_+^n)^{\Hom(F_v,\Qbar_p)}$ and an inertial type $\tau$ defined over $E$.
There is a (possibly zero) $\calO$-flat reduced quotient $R_v^{\lambda,\tau,\cris}$ of $R_v^\square$ such that an $E$-algebra morphism $x : R_v^\square[1/p] \rightarrow \Qbar_p$ factors through $R_v^{\lambda,\tau,\cris}[1/p]$, if and only if $\rho_x$ is potentially crystalline of weight $\lambda$ inertial type $\tau$.
If nonzero, then $\Spec R_v^{\lambda,\tau,\cris}[1/p]$ is equidimensional of dimension $n^2+\frac{n(n-1)}{2}[F_v:\Q_p]$, and is formally smooth over $E$.
\end{thm}

If $\tau = 1$, then we omit it from the notation, and just write $R_v^{\lambda,\mathrm{cr}}$ for  $R_v^{\lambda,1,\mathrm{cr}}$.
This theorem yields the following corollary, using the argument \cite{EmertonGeeBM}*{\S4.3}. 

\begin{cor}\label{thm:crdefringdet}
Assume $p\nmid n$. 
Fix $\lambda \in (\Z_+^n)^{\Hom(F_v,\Qbar_p)}$ and an inertial type $\tau$ defined over $E$.
There is a (possibly zero) $\calO$-flat reduced quotient $R_v^{\lambda,\tau,cris,\mu}$ of $R_v^{\square,\mu}$ such that an $E$-algebra morphism $x : R_v^{\square,\mu}[1/p] \rightarrow \Qbar_p$ factors through $R_v^{\lambda,\tau,\cris,\mu}[1/p]$ if and only if $\rho_x$ is potentially crystalline of weight $\lambda$ and inertial type $\tau$.
If nonzero, then $\Spec R_v^{\lambda,\tau,\cris,\mu}[1/p]$ is equidimensional of dimension $n^2-1 + \frac{n(n-1)}{2}[F_v:\Q_p]$ and is formally smooth over $E$. 
\end{cor}

In \cite{EmertonGeeBM}*{Lemma~4.3.1} it is assumed $p>n$, but this is only used to guarantee that a character, denoted $\theta$ there, has an $n$th root. 
For this it suffices that for any positive integer $k$, the binomial coefficient $\binom{1/n}{k}$ is a $p$-adic integer, which holds whenever $p\nmid n$.

Note that knowing whether or not the above rings are nonzero amounts to showing there is a potentially semistable or potentially crystalline lift with the required weight and inertial type (and determinant). 
In general, this seems to be a difficult problem. 

We recall the notion of a potentially diagonalizable representation from \cite{BLGGT}. 
We say a lift
	\[ \rho : G_v \lra \GL_n(\calO_{\Qbar_p}) \]
of $\rhobar_v\otimes \Fbar_p$ is \emph{potentially diagonalizable of weight} $\lambda \in (\Z_+^n)^{\Hom(F_v,\Qbar_p)}$ if there is a finite extension $K/F_v$ and continuous characters $\chi_1,\ldots,\chi_n : G_K \lra \calO_{\Qbar_p}$, such that the following hold:
\begin{itemize}
	\item $\rho|_{G_K}$ and $\chi_1\oplus\cdots\oplus \chi_n$ are both crystalline of weight $\lambda_K$, where $\lambda_K \in (\Z_+^n)^{\Hom(K,\Qbar_p)}$ is given by $\lambda_{K,\sigma'} = \lambda_\sigma$ if $\sigma' : K \hookrightarrow \Qbar_p$ extends $\sigma : F_v \hookrightarrow \Qbar_p$;
	\item the $\calO_{\Qbar_p}$-points of $\Spec R_{\rhobar_v|_{G_K}}^{\lambda,\cris}$ determined by $\rho|_{G_K}$ and $\chi_1\oplus \cdots \oplus \chi_n$ lie on the same irreducible component.
	\end{itemize}
We say $\rho$ is \emph{potentially diagonalizable of regular weight} if there is $\lambda \in (\Z_+^n)^{\Hom(F_v,\Qbar_p)}$ such that $\rho$ is potentially diagonalizable of weight $\lambda$. 
One can also define potentially diagonalizable lifts in the nonregular weight case, but we will not have use for them here.
We also note that in \cite{BLGGT}, potentially diagonalizable is defined using irreducible components of $\Spec R_{\rhobar_v|_{G_K}}^{\lambda,\cris}\otimes \Qbar_p$, but it is equivalent to use irreducible components of $\Spec R_{\rhobar_v|_{G_K}}^{\lambda,\cris}$, postcomposing $\chi_1\oplus\cdots\oplus\chi_n$ by an element of $\Gal(\Qbar_p/E)$ if necessary. 

For given $\lambda \in (\Z_+^n)^{\Hom(F_v,\Qbar_p)}$ and inertial type $\tau$ defined over $E$, we'll call an $\calO_{\Qbar_p}$-point $x$ of $\Spec R_v^{\lambda,\tau,\cris}$, or of $\Spec R_v^{\lambda,\tau,\cris,\mu}$, \emph{potentially diagonalizable} if $\rho_x$ is. 
We will call an irreducible component of $\Spec R_v^{\lambda,\tau,\cris}$, or of $\Spec R_v^{\lambda,\tau,\cris,\mu}$, \emph{potentially diagonalizable} if it contains a potentially diagonalizable point.

Finally, we note that by \cite{BLGGT}*{Lemma~1.4.1}, the notion of being potentially diagonalizable depends only on $\rho \otimes \Qbar_p$. 
Hence, it makes sense to talk about a $\Qbar_p$-valued representations as being potentially diagonalizable without specifying an invariant lattice. 
This in particular applies to the restrictions to $G_v$ of the automorphic Galois representations of \S\ref{sec:AutGalRep}.

It is not currently known whether or not every potentially crystalline representation is potentially diagonalizable, nor whether or not any residual representation has a potentially diagonalizable lift of regular weight. 
The latter question has been investigated by Gee--Herzig--Liu--Savitt \cite{GHLS}. 
In particular, in dimension two it is always possible, and we will use the following lemma in \S\ref{sec:Hilb}.
\begin{lem}\label{thm:BTlift}
Assume that $n = 2$, and that $\mu$ is de~Rham.
Then $\rhobar_v$ admits a regular weight potentially diagonalizable lift with determinant $\mu$. 
\end{lem}

\begin{proof}
First assume $\rhobar_v$ is peu ramifi\'{e}e, in the sense of \cite{GHLS}*{Definition~2.1.3}.
Fix $\lambda \in (\Z_+^2)^{\Hom(F_v,\Qbar_p)}$ such that $\lambda_{\sigma,1} + 1 + \lambda_{\sigma,2} = \HT_\sigma(\mu)$ for each $\sigma\in \Hom(F_v,\Qbar_p)$. 
By \cite{GHLS}*{Corollary~2.1.11}, $\rhobar_v$ has a potentially diagonalizable lift $\rho : G_v \lra \GL_2(\calO_{\Qbar_p})$ of weight $\lambda$. 
Enlarging $E$ if necessary, we can assume $\det\rho$ takes values in $\calO^\times$, and 
$\mu(\det\rho)^{-1} : G_v \rightarrow 1+\frakm_{\calO}$ is finitely ramified by choice of $\lambda$. 
Since $p>2$, there is a continuous finitely ramified character $\eta : G_v \rightarrow 1+\frakm_{\calO}$ such that $\eta^2 = \mu(\det\rho)^{-1}$. 
Then $\rho\otimes\eta$ is a potentially diagonalizable lift of $\rhobar$ of weight $\lambda$ and determinant $\mu$.

It only remains to treat the case (see \cite{GHLS}*{Examples~2.1.4})
	\[ \rhobar \cong \begin{pmatrix}
	\overline{\chi} & c \\ & \overline{\chi}^{-1}\overline{\mu}
	\end{pmatrix} \]
with $\overline{\chi}^2\mubar^{-1} = \epsilonbar$, and tr\`{e}s ramifi\'{e}e cocycle $c$.
The result the follows from the argument of \cite{BLGGOrdLifts1}*{Lemma~6.1.6}. 
We sketch the details.
Twisting, we may assume that $\HT_\sigma(\mu) > 1$ for each $\sigma : F_v\hookrightarrow \Qbar_p$. 
Then any lift of the form
	\[ \rho \cong \begin{pmatrix}
	\chi & \ast \\ & \chi^{-1}\mu 
	\end{pmatrix} \]
with $\chi$ a finitely ramified character lifting $\overline{\chi}$ is potentially crystalline, hence potentially diagonalizable by \cite{BLGGT}*{Lemma~1.4.3}.

Choose some finitely ramified lift $\chi$ of $\chibar$ and set $\psi:=\chi^2\mu^{-1}\epsilon^{-1}$. 
Let $L$ be the line in $H^1(G_v,\F(\epsilonbar))$ spanned by the cohomology class generated by $c$, and let $H$ be the hyperplane in $H^1(G_v,\F)$ annihilated by $L$ under local Tate duality. 
Fix a uniformizer $\varpi$ in $\calO$. 
Then there is $n\ge 1$ such that $\psi^{-1} \mod {\varpi^{n+1}} = 1 + \varpi^n\overline{\alpha}$ with $\overline{\alpha} : G_v \rightarrow \F$ a nontrivial homomorphism. 
An argument using local Tate duality, as in see Case~2 in the proof of \cite{BLGGOrdLifts1}*{Lemma~6.1.6}, shows it that is suffices to prove $\overline{\alpha} \in H$. 
We assume otherwise. 
Since $\rhobar$ is tr\`{e}s ramifi\'{e}e, $H$ does not contain the unramfied line, and we can find $\overline{a} \in \F^\times$ such that  $\alpha + u_{\overline{a}} \in H$, where $u_{\overline{a}} : G_v \rightarrow \F$ is the unramified homomorphism that sends $\Frob_v$ to $\overline{a}$. 
Choosing some $b\in \calO^\times$ such that $\overline{a} = 2b \mod \varpi$, and replacing $\chi$ by $\chi$ times the unramified character that sends $\Frob_v$ to $(1+\varpi^n b)^{-1}$ yields $\overline{\alpha} \in H$, finishing the proof. 
\end{proof}

Let $\rho : G_v \lra \GL_n(\Qbar_p)$ 
be a continuous representation. 
We say $\rho$ is \emph{ordinary of weight} $\lambda \in (\Z_+^n)^{\Hom(F_v,\Qbar_p)}$ if there is a $G_v$-stable decreasing filtration
	\[ \Qbar_p^n = V_1 \supset V_2 \supset \cdots \supset V_{n+1} = \{0\} \]
with one dimensional graded pieces and an open subgroup $U$ of $\calO_{F_v}^\times$ such that for each $1\le j \le n$, 
letting $\psi_j : G_v \rightarrow \Qbar_p^\times$ denote the character giving the $G_v$-action on $V_j/V_{j+1}$, we have
	\[  \psi_j\circ\Art_{F_v}(x) = \prod_{\sigma: F_v \hookrightarrow \Qbar_p} \sigma (x)^{j - n -\lambda_{\sigma,j}} \]
for all $x \in U$. 
We say $\rho$ is \emph{ordinary} of \emph{regular weight} if there is $\lambda\in (\Z_+^n)^{\Hom(F_v,\Qbar_p)}$ such that it is ordinary of weight $\lambda$. 
The following is a result of Geraghty \cite{GeraghtyOrdinary}*{Lemma~3.3.3}.

\begin{thm}\label{thm:ordring}
Let $\lambda \in (\Z_+^n)^{\Hom(F_v,\Qbar_p)}$ and let $\tau$ be an inertial type defined over $E$.
There is a (possibly zero)  $\calO$-flat reduced quotient $R_v^{\lambda,\tau,\ord}$ of $R_v^\square$ such that 
an $E$-algebra morphism $x : R_v^\square[1/p] \rightarrow \Qbar_p$ factors through $R_v^{\lambda,\tau,\ord}[1/p]$ if and only if $\rho_x$ is ordinary of weight $\lambda$.
If nonzero, then $\Spec R_v^{\lambda,\tau,\ord}[1/p]$ 
is equidimensional of dimension $n^2 + \frac{n(n-1)}{2}[F_v:\Q_p]$ and admits an open dense subscheme that is formally smooth over $E$.
\end{thm}

Note that in order for $R_v^{\lambda,\tau,\ord}$ to be nonzero, $\tau$ must be a direct sum of finite order characters of $I_v$. 

\section{Automorphic Galois representations}\label{sec:AutGalRep}

We recall the notion of regular algebraic polarized cuspidal automorphic representations and their associated Galois representations. 
Throughout this section, $F$ is assumed to be either a CM or totally real number field, with maximal totally real subfield $F^+$. 
Let $c \in G_{F^+}$ be a choice of complex conjugation. 


\subsection{Polarized auotmorphic Galois representations}\label{sec:pol} 
Following \cite{BLGGT}*{\S2.1}, we say that a pair $(\pi,\chi)$ is a \emph{polarized automorphic representation} of $\GL_n(\A_F)$ if
\begin{itemize}
	\item $\pi$ is an automorphic representation of $\GL_n(\A_F)$;
	\item $\chi : (F^+)^\times \backslash \A_{F^+}^\times \rightarrow \C^\times$ is a continuous character with $\chi_v(-1)$ independent of $v | \infty$;
	\item $\pi^c \cong \pi^\vee \otimes (\chi\circ \Nm_{F/F^+}\circ \det)$;
	\end{itemize}
When $F$ is totally real, the requirement that $\chi_v(-1)$ be independent of $v|\infty$ has been shown to be redundant by Patrikis \cite{PatrikisSign}*{Theorem~2.1}. 
Note that when $F$ is CM, if $(\pi,\chi)$ is a polarized automorphic representation of $\GL_n(\A_F)$, then so is $(\pi,\chi\delta_{F/F^+})$.
We do not specify a sign convention in this generality, unlike in \cite{BLGGT}*{\S2.1}, but will in the regular algebraic case below. 
Our convention will sometimes differ from that of \textit{loc. cit.}, 
but it ensures that the character $\chi_\iota \epsilon^{1-n}$ of $G_{F^+}$ in \cref{thm:autgalrep} below is totally odd, 
which is more convenient for us (see \cref{thm:Gdhoms,thm:polringdim}).

We say that an automorphic representation $\pi$ of $\GL_n(\A_F)$ is \emph{polarizable} if there is a character $\chi$ such that $(\pi,\chi)$ is a polarized automorphic representation.	
If $F$ is CM and 
$(\pi,\delta_{F/F^+}^n)$ is polarized, then we say that $\pi$ is \emph{conjugate self-dual}. 
An automorphic representation $\pi$ of $\GL_n(\A_F)$ is called \emph{regular algebraic} if $\pi_\infty$ has the same infinitesimal character as an irreducible algebraic representation of $\Res_{F/\Q} \GL_n$. 
If $\lambda = (\lambda_\sigma) \in (\Z_+^n)^{\Hom(F,\C)}$, then we let $\xi_\lambda$ denote the irreducible algebraic representation of $\prod_{\sigma}\GL_n$ which is the tensor product over $\sigma \in \Hom(F,\C)$ of the irreducible algebraic representations with highest weight $\lambda_\sigma$. 
We say a regular algebraic automorphic representation $\pi$ of $\GL_n(\A_F)$ has \emph{weight} $\lambda \in (\Z_+^n)^{\Hom(F,\C)}$ if $\pi_\infty$ has the same infinitesimal character as $\xi_\lambda^\vee$. 
We will say a polarized automorphic representation $(\pi,\chi)$ of $\GL_n(\A_F)$ is cuspidal if $\pi$ is. 
We will say a polarized automorphic representation $(\pi,\chi)$ of $\GL_n(\A_F)$ is regular algebraic if $\pi$ is. 
In this case $\chi$ is necessarily an algebraic character, and if $F$ is CM, we fix the sign of $\chi_v(-1)$ for $v | \infty$ as follows. 
Letting $q\in \Z$ be the unique integer such that $\chi \lvert \cdot \rvert_{\A_{F^+}}^q$ is finite order, we require $\chi_v(-1) = (-1)^{n+q}$ for each $v|\infty$ in $F^+$.

Let $\pi$ be a regular algebraic polarizable cuspidal automorphic representation of $\GL_n(\A_F)$ of weight $\lambda$, and let $\iota : \Qbar_p \xrightarrow{\sim} \C$ be an isomorphism. 
Let $v|p$ in $F$ and let $\varpi_v$ be a choice of uniformizer for $F_v$. 
For any integer $a\ge 1$, let $\Iw(v^{a,a})$ denote the subgroup of $\GL_n(\calO_{F_v})$ of matrices that reduce to an upper triangular matrix modulo $\varpi_v^a$. 
For each $1\le j \le n$, the space $(\iota^{-1}\pi_v)^{\Iw(v^{a,a})}$ has an action of the Hecke operator
	\[ U_{\varpi_v}^{(j)} = \left[ \Iw(v^{a,a}) \begin{pmatrix} \varpi_v 1_j & \\ & 1_{n-j} \end{pmatrix}
	\Iw(v^{a,a}) \right], \]
and they commute with one another. 
We define modified Hecke operators 
	\[ U_{\lambda,\varpi_v}^{(j)} = \Big(\prod_{\sigma : F_v \hookrightarrow \Qbar_p} \prod_{i=1}^j
	\sigma(\varpi_v)^{-\lambda_{\iota\sigma,n-i+1}}\Big) U_{\varpi_v}^{(j)}\]
for each $1\le j \le n$. 
We say that $\pi$ is $\iota$-\emph{ordinary} if for each $v|p$, there is an integer $a \ge 1$ and a nonzero vector in $(\iota^{-1}\pi_v)^{\Iw(v^{a,a})}$ that is an eigenvector for each $U_{\lambda,\varpi_v}^{(1)},\ldots,U_{\lambda,\varpi_v}^{(n)}$ with eigenvalues that are $p$-adic units. 
This definition does not depend on the choice of $\varpi_v$. 
We say a polarized regular algebraic cuspidal automorphic representation $(\pi,\chi)$ of $\GL_n(\A_F)$ is $\iota$-\emph{ordinary} if $\pi$ is.


The following theorem is due to the work of many people. 
We refer the reader to \cite{BLGGT}*{Theorem~2.1.1} and the references contained there for \cref{autgalrep:pol,autgalrep:pss,autgalrep:locglob} (noting that the assumption of an Iwahori fixed vector in part (4) of \cite{BLGGT}*{Theorem~2.1.1} can be removed by the main result of \cite{Caraianilp}), and to \cite{ThorneReducible}*{Theorem~2.4} for \cref{autgalrep:ord}.


\begin{thm}\label{thm:autgalrep}
Let $(\pi,\chi)$ be a regular algebraic, polarized, cuspidal automorphic representation of $\GL_n(\A_F)$, of weight $\lambda \in (\Z_+^n)^{\Hom(F,\C)}$. 
Fix an isomorphism $\iota : \Qbar_p \xrightarrow{\sim} \C$, and for each $v|p$ in $F$, let $\lambda_v = (\lambda_{v,\sigma}) \in (\Z_+^n)^{\Hom(F_v,\Qbar_p)}$ be given by $\lambda_{v,\sigma} = \lambda_{\iota\sigma}$. 
There is a continuous semisimple representation 
	\[ \rho_{\pi,\iota} : G_F \lra \GL_n(\Qbar_p) \]
satisfying the following properties. 
\begin{enumerate}[ref=\arabic*]
\item\label{autgalrep:pol} 
There is a perfect symmetric pairing $\langle \cdot, \cdot \rangle$ on $\Qbar_p^n$ 
such that for any $a,b\in \Qbar_p^n$ and $\gamma\in G_F$,
	\[ \langle \rho_{\pi,\iota}(\gamma) a, \rho_{\pi,\iota}(c\gamma c) b \rangle = (\chi_\iota\epsilon^{1-n})(\gamma)\langle a, b \rangle. \] 
\item\label{autgalrep:pss} For all $v|p$, $\rho_{\pi,\iota}|_{G_v}$ is potentially semistable of weight $\lambda_v$.
\item\label{autgalrep:locglob} For any finite place $v$, 
	\[\iota\WD(\rho_{\pi,\iota}|_{G_v})^{\Fss} \cong \rec_{F_v}(\pi_v\otimes \abs{\cdot}^{\frac{1-n}{2}}). \]
\item\label{autgalrep:ord} If $(\pi,\chi)$ is $\iota$-ordinary, then for each $v|p$,  $\rho_{\pi,\iota}$ is ordinary of weight $\lambda_v$.
\end{enumerate}
\end{thm}

In \cref{autgalrep:locglob}, $\rec_{F_v}$ is the Local Langlands reciprocity map that takes irreducible admissible representations of $\GL_n(F_v)$ to Frobenius semi-simple Weil--Deligne representations, normalized as in \cite{HarrisTaylor} and \cite{HenniartLL}, and $\mathrm{WD}(\rho_{\pi,\iota}|_{G_v})$ is the $\Qbar_p$-Weil--Deligne representation associated to $\rho_{\pi,\iota}|_{G_v}$.

Since $G_{F,S}$ is compact, $\rho_{\pi,\iota}(G_{F,S})$ stabilizes a lattice. 
So conjugating $\rho_{\pi,\iota}$ if necessary, we may assume it takes valued in $\GL_n(\calO_{\Qbar_p})$, and denote by 
	\[ \rhobar_{\pi,\iota} : G_{F,S} \lra \GL_n(\Fbar_p)\]
the semisimplification of its reduction modulo $\frakm_{\Qbar_p}$, which is independent of the choice of lattice.

\subsection{Hilbert modular forms}\label{sec:HilbGalReps}
If $n = 2$ and $F$ is totally real, then any automorphic representation $\pi$ of $\GL_2(\A_F)$ is polarizable. 
More specifically, if $\chi$ denotes the central character of $\pi$, then $(\pi,\chi)$ is polarized. 

If $\pi$ is regular algebraic and cuspidal, then $\pi$ is a twist of the automorphic representation generated by a Hilbert modular cusp form (see \cite{ClozelRegAlg}*{\S 1.2.3}). 

\section{Polarized deformation rings and automorphic points}\label{sec:CM}

In this section, we prove our main theorems on polarized deformation rings for Galois representations of CM fields. 
Before proceeding, we fix the assumptions and notation that will be used throughout this section. 

\subsection{Setup}\label{sec:polarsetup}
We assume $p>2$, and fix an isomorphism $\iota : \Qbar_p \xrightarrow{\sim} \C$. 
We assume that our fixed number field $F$ is CM, and denote by $F^+$ its maximal totally real subfield. 
Fix a finite set of places $S$ of $F^+$ containing all places above $p$, and let $F_S$ be the maximal extension of $F$ unramified outside of the places in $F$ above $S$. 
Note that $F_S/F^+$ is Galois, and we set $G_{F^+,S} = \Gal(F_S/F^+)$ and $G_{F,S} = \Gal(F_S/F)$.
We fix a choice of complex conjugation $c\in G_{F^+}$. 
We fix a continuous absolutely irreducible 
	\[ \rhobar : G_{F,S} \lra \GL_n(\F) \]
and a continuous totally odd character $\mu : G_{F^+,S} \rightarrow\calO^\times$ and assume there is a perfect symmetric pairing $\langle \cdot, \cdot \rangle$ on $\F^n$ such that
	\[ \langle \rhobar(\gamma) a, \rhobar(c\gamma c)b \rangle = \mu(\gamma) \langle a, b \rangle \]
for all $\gamma \in G_{F,S}$ and $a,b \in \F^n$,  
where $\mubar : G_{F^+,S} \rightarrow \F^\times$ is the reduction of $\mu$ modulo $\frakm_{\calO}$. 
Let $R^{\pol}$ be the universal $\mu$-polarized deformation ring for $\rhobar$ as in \S\ref{sec:polardefring}.
Since the pairing $\langle \cdot, \cdot \rangle$ is symmetric, we can and do fix a lift
	\[ \rbar : G_{F^+,S} \lra \calG_n(\F) \]
of $\rhobar$ with $\nu \circ \rbar = \mubar$ as in \cref{thm:Gdhoms}. 
We also view $R^{\pol}$ as the universal $\mu$-polarized deformation ring of $\rbar$.

\begin{lem}\label{thm:polringdim}
There is a presentation
	\[ R^{\pol} \cong \calO[[x_1,\ldots,x_g]]/(f_1,\ldots,f_k)\]
with $g-k \ge \frac{n(n+1)}{2}[F^+:\Q]$. 
In particular, every irreducible component of $\Spec R^{\pol}$ has dimension at least $1+\frac{n(n+1)}{2}[F^+:\Q]$.
\end{lem}

\begin{proof}
The result then follows from \cref{thm:genpolpres} and the global Euler--Poincar\'{e} characteristic formula, using that $H^0(G_{F^+,S},\ad(\rbar)) = 0$ since $\rhobar$ is absolutely irreducible, and that $\dim_{\F} \ad(\rbar)^{c_v = 1} = \frac{n(n-1)}{2}$ for each $v|\infty$ since $\mu$ is totally odd (see \cite{CHT}*{Lemma~2.1.3}). 
\end{proof}

\begin{defn}\label{def:autpoint}
Let $R$ be a quotient of $R^{\pol}$, let $x \in \Spec R(\Qbar_p)$, and let $\rho_x$ be the pushforward of the universal $\mu$-polarized deformation via $R^{\pol} \twoheadrightarrow R \xrightarrow{x} \Qbar_p$.
We call $x$ an \emph{automorphic point} if there is a regular algebraic polarized cuspidal automorphic representation $(\pi,\chi)$ of $\GL_n(\A_F)$ such that $\rho_x \cong \rho_{\pi,\iota}$ and $\mu = \chi_\iota \epsilon^{1-n}$.

Given a finite extension $L/F$ of CM fields, we say $x$ is an $L$-\emph{potentially automorphic point} if there is a regular algebraic polarized cuspidal automorphic representation $(\pi,\chi)$ of $\GL_n(\A_L)$ such that $\rho_x|_{G_L} \cong \rho_{\pi,\iota}$ and $\mu|_{G_{L^+}} = \chi_\iota \epsilon^{1-n}$.

In either case, we say $x$ is $\iota$-\emph{ordinary} if $\pi$ is, and that $x$ has \emph{level prime to} $p$, resp. \emph{potentially prime to} $p$, if for all $v|p$ the local representation $\pi_v$ is unramified, resp. becomes unramified after a finite base change.

If $X^{\rig}$ is the rigid analytic generic fibre of $\Spf R$, and $x^{\rig} \in X^{\rig}$ is the point corresponding to $\ker(x) \subset R[1/p]$, then we say $x^{\rig}$ is an \emph{automorhpic point} if $x$ is, and if this is the case we further say $x^{\rig}$ has \emph{level prime to} $p$ if $x$ does.
\end{defn}

\subsection{Small $R = \mathbb{T}$ theorems from the literature}\label{sec:CMsmallRT}
In this subsection we recall the small $R = \bbT$ theorems that are used in the proofs of our main theorems. 
Before stating them, we recall some terminology of \cite{CHT} for the deformation theory of $\rbar$. 

Let $\tildeS$ be a finite set of finite places of $F$ such that every $w\in \tildeS$ is split over some $v\in S$, and $\tildeS$ contains at most one place above any $v\in S$. 
For each $w\in \tildeS$, let $R_w$ be a quotient of $R_w^\square : = R_{\rhobar|_{G_w}}^\square$ that represents a local deformation problem.
By \cref{thm:defquolem}, the rings $R_w$ in \cref{thm:PDsmallRT,thm:ordsmallRT,thm:potsmallRT} below represent local deformation problems.

We refer to the tuple
	\[ \calS = (F/F^+, S, \tildeS, \calO, \rbar, \mu, \{R_w\}_{w\in \tildeS}) \]
as a \emph{global} $\calG_n$-\emph{deformation datum}. 
This differs from the definition in \cite{CHT}*{\S2.3} in that our ramification set $S$ may contain places that do not split in $F/F^+$, and $\tildeS$ is not required to contain a place above every $v\in S$. 
As the results in \cite{ChenevierFern} make no assumption on the splitting behaviour in $F$ of the places in $S \smallsetminus \{v|p\}$, we also wish to make no such assumption.

A \emph{type} $\calS$ deformation of $\rbar$ is a deformation $r : G_{F^+,S} \rightarrow \calG_n(A)$ with $A$ a $\CNL_{\calO}$-algebra such that for any (equivalently for one) lift $r$ in its deformation class 
	\begin{itemize}
	\item $\nu \circ r = \mu$, and 
	\item for each $w\in \tildeS$, the $\CNL_{\calO}$-morphism $R_w^\square \rightarrow A$ induced by the lift $r|_{G_w}$ of $\rbar|_{G_w} = \rhobar|_{G_w}$, factors through $R_w$.
\end{itemize}
We let $D_{\calS}$ be the set valued functor on $\CNL_{\calO}$ that takes a $\CNL_{\calO}$-algebra $A$ to the set of deformations of type $\calS$. 
It is easy to see that $D_{\calS}$ is represented by a quotient $R_{\calS}$ of $R^{\pol}$. 
Indeed, let $R_{\tildeS}^\square = \widehat{\otimes}_{w\in \tildeS} R_w^\square$ and $R_{\calS}^{\loc} = \widehat{\otimes}_{w\in \tildeS} R_w$, where the completed tensor products are taken over $\calO$.
A choice of lift $r$ in the equivalence class of the universal $\mu$-polarized deformation of $\rbar$ determines a $\CNL_{\calO}$ algebra morphism $R_{\tildeS}^\square \rightarrow R^{\pol}$
by $r \mapsto \{r|_{G_w}\}_{w\in\tildeS}$. 
This $R_{\tildeS}^\square$-algebra structure on $R^{\pol}$ may depend on the choice of lift $r$, but it is canonical up to $\CNL_{\calO}$-automorphisms of $R_{\tildeS}^\square$. 
We can then define
	\[ R_{\calS} := R_{\calS}^{\loc} \otimes_{R_{\tildeS}^\square} R^{\pol}.\]
We call $R_{\calS}$ the \emph{universal type} $\calS$ \emph{deformation ring}, and note that it has an $R_{\calS}^{\loc}$-algebra structure that is canonical up to $\CNL_{\calO}$-automorphisms of $R_{\calS}^{\loc}$.
The following lemma follows immediately from the construction of $R_{\calS}$.

\begin{lem}\label{thm:typeSquotient}
Let $\tildeS' \subseteq \tildeS$, and let $\calS'$ be the global $\calG_n$-deformation datum
	\[ \calS' := (F/F^+, S, \tildeS', \calO, \rbar, \mu, \{R_w\}_{w\in \tildeS'}). \]
Let $T = \tildeS \smallsetminus \tildeS'$. 
There is a canonical isomorphism
	\[ R_{\calS} \cong R_{\calS'} \otimes_{(\widehat{\otimes}_{w\in T} R_w^\square)} 
	(\widehat{\otimes}_{w\in T} R_w).\]
In particular, if $R_w = R_w^\square$ for all $w\in \tildeS \smallsetminus \tildeS'$, then there is a canonical isomorphism $R_{\calS} \cong R_{\calS'}$.
\end{lem}

Before proceeding, we introduce the conditions on the residual representation that appear as assumptions in the small $R = \bbT$ theorems we quote.
We first recall the definition of an adequate subgroup of $\GL_n(\F)$ in \cite{GHTadequateIM}*{\S1}.

\begin{defn}\label{def:GLadequate}
A subgroup $\Gamma$ of $\GL_n(\Fbar)$ is \emph{adequate} if the following hold:
\begin{enumerate}
\item $H^1(\Gamma,\Fbar) = 0$ and $H^1(\Gamma,\frakgl_n(\Fbar)/\frakz) = 0$, where $\frakz$ is the centre of $\frakgl_n(\Fbar)$;
\item\label{GLadequate:End} $\End_{\Fbar}(\Fbar^n)$ is spanned by the semisimple elements in $\Gamma$.
\end{enumerate}
\end{defn}

This is slightly more general than \cite{ThorneAdequate}*{Definition~2.3}, as it allows $p|n$. 
However, if $p\nmid n$ then the two definitions are equivalent.
We note that \cref{GLadequate:End} implies that $\Gamma$ acts irreducibly on $\Fbar^n$. 
The following partial converse is a theorem of Guralnick--Herzig--Taylor--Thorne \cite{ThorneAdequate}*{Theorem~A.9}.

\begin{thm}\label{thm:adequate}
Let $\Gamma$ be a subgroup of $\GL_n(\Fbar)$ that acts absolutely irreducibly on $\Fbar^n$. 
Let $\Gamma^0$ be the subgroup of $\Gamma$ generated by elements of $p$-power order. 
Let $d\ge 1$ be the maximal dimension of an irreducible $\Gamma^0$-submodule of $\Fbar^n$.
If $p > 2(d+1)$, then $\Gamma$ is adequate and $p\nmid n$.
\end{thm}

We now state the small $R = \bbT$ theorems that we use. 
The first of which is due to Barnet-Lamb--Gee--Geraghty--Taylor \cite{BLGGT}, with improvements by Barnet-Lamb--Gee--Geraghty \cite{BLGGU2}*{Appendix~A} and Dieulefait--Gee \cite{DieulefaitSym5}*{Appendix~B}.

\begin{thm}\label{thm:PDsmallRT}
Assume that $p\nmid 2n$ and that every $v|p$ in $F^+$ splits in $F$. 
For each $v|p$ in $F^+$, fix a choice of place $\tilde{v}$ of $F$ above $v$, and set $\tildeS_p = \{\tilde{v}\}_{v|p \text{ in }F^+}$. 
Assume further:
\begin{ass}
\item $\mu$ is de~Rham.
\item\label{PDsmallRT:PD} $\rhobar \otimes \Fbar_p \cong \rhobar_{\pi,\iota}$ and $\mubar = \chi_\iota\epsilon^{1-n} \mod {\frakm_{\Qbar_p}}$, where $(\pi,\chi)$ is a regular algebraic polarized cuspidal automorphic representation of $\GL_n(\A_F)$ such that $\rho_{\pi,\iota}|_{G_{\tilde{v}}}$ is potentially diagonalizable for each $\tilde{v} \in \tildeS_p$.
\item\label{PDsmallRT:adequate} $\rhobar(G_{F(\zeta_p)})$ is adequate and $\zeta_p \notin F$.
\end{ass}
For each $\tilde{v} \in \tildeS_p$, fix $\lambda_{\tilde{v}} \in (\Z_+^n)^{\Hom(F_{\tilde{v}},\Qbar_p)}$ and an inertial type $\tau_{\tilde{v}}$ defined over $E$, and 
we let $R_{\tilde{v}}$ be a quotient of $R_{\tilde{v}}^{\lambda_{\tilde{v}},\tau_{\tilde{v}},\cris}$ corresponding to a union of potentially diagonalizable irreducible components of $\Spec R_{\tilde{v}}^{\lambda_{\tilde{v}},\tau_{\tilde{v}},\cris}$. 
Let $\calS$ be the global $\calG_n$-deformation datum 
	\[ \calS = (F/F^+, S, \tildeS_p, \calO, \rbar, \mu, \{R_{\tilde{v}}\}_{\tilde{v}\in \tildeS_p}). \]
Then the following hold:
\begin{enumerate}
\item\label{PDsmallRT:finite} The universal type $\calS$ deformation ring $R_{\calS}$ is finite over $\calO$.
\item\label{PDsmallRT:autpt} Every $x\in \Spec R_{\calS}(\Qbar_p)$ is automorphic of level potentially prime to $p$.
\end{enumerate} 
\end{thm}

\begin{proof}
\Cref{PDsmallRT:autpt} follows from \cite{DieulefaitSym5}*{Theorem~9}. 
To show \cref{PDsmallRT:finite}, we will apply \cite{ThorneAdequate}*{Theorem~10.1}. 
However, we have not fixed irreducible components at the finite places $\tilde{v}\in \tildeS$.
We now explain how we can reduce to this case using base change and \cite{BLGGU2}*{Theorem~6.8}. 
Note that both \cite{DieulefaitSym5}*{Appendix~B} and \cite{BLGGU2}*{Theorem~6.8} use the stronger definition of adequate that implies $p\nmid n$.

To prove that $R_{\calS}$ is finite over $\calO$, it suffices to prove that $R_{\calS}^{\red}$ is finite over $\calO$. 
For this, it suffices to prove that $R_{\calS}/\frakq$ is finite over $\calO$ for any minimal prime $\frakq$ of $R_{\calS}$, since $R_{\calS}^{\red}$ injects into $\prod_{\frakq} R_{\calS}/\frakq$. 
Fix a minimal prime $\frakq$ of $R_{\calS}$, and let $r_{\frakq} : G_{F^+,S} \rightarrow \calG_n(R_{\calS}/\frakq)$ be the induced deformation. 

We now choose a finite solvable extension $L/F$ of CM fields, with maximal totally real subfield $L^+$, satisfying the following:
\begin{itemize}
	\item $L$ is disjoint from the subfield of $\overline{F}$ fixed by $\rbar|_{G_{F(\zeta_p)}}$.
	\item Letting $S_{L^+}$ denote the set of places in $L^+$ above those in $S$, each $w \in S_{L^+}$ splits in $L$. 
	For each $w\in S_{L^+}$ we fix a choice of place $\tilde{w}$ in $L$ above $w$ such that if $w|p$, then $\tilde{w}$ lies above some $\tilde{v} \in \tildeS_p$.
	We let $\tildeS_L = \{ \tilde{w} \mid w \in S_{L^+}\}$. 
	\item For each $\tilde{w} \in \tildeS_L$ with $\tilde{w} \nmid p$, the $\CNL_{\calO}$-algebra map $R_{\tilde{w}}^\square \rightarrow R_{\calS}/\frakq$ induced by $r_{\frakq}|_{G_{\tilde{w}}}$ factors through the quotient $R_{\tilde{w}}^1$ of \cref{thm:R1} (here we use \cref{thm:R1factor}).
	\item For each $\tilde{w} \in \tildeS_L$, $\tau_{\tilde{v}}|_{G_{\tilde{w}}} = 1$, where $\tilde{v} \in \tildeS_p$ is the place in $\tildeS_p$ below $\tilde{w}$.
\end{itemize}
We define a global $\calG_n$-deformation datum
	\[ \calS_L = (L/L^+, S_{L^+}, \tildeS_L, \calO, \rbar|_{G_{L^+}}, \mu|_{G_{L^+}}, \{R_{\tilde{w}}\}_{\tilde{w}\in \tildeS_L})\]
where
\begin{itemize}
	\item for $\tilde{w}\nmid p$, $R_{\tilde{w}}$ is a quotient of $R_{\tilde{w}}^1$ by a minimal prime through which the $\CNL_{\calO}$-algebra morphism $R_{\tilde{w}}^1 \rightarrow R_{\calS}/\frakq$ induced by $r_{\frakq}|_{G_{\tilde{w}}}$ factors;
	\item for $\tilde{w} \mid p$, $R_{\tilde{w}}$ is a quotient of $R_{\tilde{w}}^{\lambda_{\tilde{w}},\cris}$ through which the $\CNL_{\calO}$-algebra morphism 
		\[ R_{\tilde{w}}^{\lambda_{\tilde{w}},\cris}\lra R_{\tilde{v}} \lra R_{\calS}/\frakq\] 
	induced by $r_{\frakq}|_{G_{\tilde{w}}}$ factors. 
	Here $\lambda_{\tilde{w},\sigma'} = \lambda_{\tilde{v},\sigma}$ if $\sigma' : L_{\tilde{w}} \hookrightarrow \Qbar_p$ extends $\sigma : F_{\tilde{v}} \hookrightarrow \Qbar_p$.
\end{itemize}
The deformation $r_{\frakq}|_{G_{L^+}}$ is of type $\calS_L$, so there is an induced $\CNL_{\calO}$-algebra map $R_{\calS_L} \rightarrow R_{\calS}/\frakq$, and it is finite by \cite{BLGGT}*{Lemma~1.2.3}. 
So we are reduced to showing $R_{\calS_L}$ is finite over $\calO$.

For $\tilde{w} \nmid p$, $R_{\tilde{w}}$ has characteristic zero by \cref{thm:R1}, so contains a $\Qbar_p$-point. 
For each $\tilde{v} \in \tildeS_p$, $R_{\tilde{w}}$ contains a potentially diagonalizable point by choice of $R_{\tilde{w}}$ and $R_{\tilde{v}}$, where $\tilde{v}$ is the place of $F$ below $\tilde{w}$.
We now apply \cite{BLGGU2}*{Theorem~6.8}, and we have a lift
	\[ r : G_{L^+,S_{L^+}} \lra \calG_n(\calO_{\Qbar_p}) \]
that defines a type $\calS_L$-deformation of $\rbar|_{G_{L^+}}$.
Then \cite{DieulefaitSym5}*{Appendix~B} implies there is a regular algebraic polarized cuspidal automorphic representation $(\pi',\chi')$ of $\GL_n(\A_L)$ such that $r|_{G_L}\otimes\Qbar_p \cong \rho_{\pi',\iota}$ and $\chi'_\iota \epsilon^{1-n} = \mu|_{G_{L^+}}$. 
 
Take $\tilde{w} \in \tildeS_L$. 
If $\tilde{w} | p$, then $\Spec R_{\tilde{w}}[1/p]$ is the unique irreducible component of $\Spec R_{\tilde{w}}^{\lambda_{\tilde{w}},\cris}[1/p]$ containing the point determined by $r|_{G_{\tilde{w}}}$, as $R_{\tilde{w}}^{\lambda_{\tilde{w}},\cris}[1/p]$ is formally smooth. 
If $\tilde{w} \nmid p$, then $\Spec R_{\tilde{w}}[1/p]$ is the unique irreducible component of $\Spec R_{\tilde{w}}^\square[1/p]$ containing $r|_{G_{\tilde{w}}}$ by \cite{BLGGT}*{Lemma~1.3.2}, using local-global compatibility, i.e. \cref{autgalrep:locglob} of \cref{thm:autgalrep}, and the fact that $\pi'$ being cuspidal implies that $\pi_w$ is generic for all finite $w$. 
We can now apply \cite{ThorneAdequate}*{Theorem~10.1}, using \cite{Thorne2adic}*{Proposition~7.1} in place of \cite{ThorneAdequate}*{Proposition~4.4} (see \cite{Thorne2adic}*{\S 7}). 
This completes the proof.
\end{proof}

For our purposes below, it would suffice to fix irreducible components at places dividing $p$ in \cref{thm:PDsmallRT} above, but we do not want to fix irreducible components at places not dividing $p$ (see \cref{rmk:minimal} below).

We have a similar theorem in the ordinary case, due to Geraghty \cite{GeraghtyOrdinary} and Thorne \cite{ThorneAdequate}.

\begin{thm}\label{thm:ordsmallRT}
Assume $p>2$ and that every $v|p$ in $F^+$ splits in $F$. 
For each $v|p$ in $F^+$, fix a choice of place $\tilde{v}$ of $F$ above $v$, and set $\tildeS_p = \{\tilde{v}\}_{v|p \text{ in }F^+}$. 
Assume further:
\begin{ass}
\item $\mu$ is de~Rham.
\item\label{ordsmallRT:ord} $\rhobar \otimes \Fbar_p \cong \rhobar_{\pi,\iota}$ and $\mubar = \chi_\iota\epsilon^{1-n} \mod {\frakm_{\Qbar_p}}$, where $(\pi,\chi)$ is an $\iota$-ordinary regular algebraic polarized cuspidal automorphic representation of $\GL_n(\A_F)$.
\item $\rhobar(G_{F(\zeta_p)})$ is adequate and $\zeta_p \notin F$.
\end{ass}
For each $\tilde{v} \in \tildeS_p$, fix $\lambda_{\tilde{v}} \in (\Z_+^n)^{\Hom(F_{\tilde{v}},\Qbar_p)}$ and an inertial type $\tau_{\tilde{v}}$ defined over $E$, and let $R_{\tilde{v}}$ be a quotient of $R_{\tilde{v}}^{\lambda_{\tilde{v}},\tau_{\tilde{v}},\ord}$ corresponding to a union of irreducible components of $\Spec R_{\tilde{v}}^{\lambda_{\tilde{v}},\tau_{\tilde{v}},\ord}$.
Let $\calS$ be the global $\calG_n$-deformation datum 
	\[ \calS = (F/F^+, S, \tildeS_p, \calO, \rbar, \mu, \{R_{\tilde{v}}\}_{\tilde{v}\in \tildeS_p}). \]
Then the following hold:
\begin{enumerate}
\item The universal type $\calS$ deformation ring $R_{\calS}$ is finite over $\calO$.
\item Every $x\in \Spec R_{\calS}(\Qbar_p)$ is $\iota$-ordinary automorphic.
\end{enumerate} 
\end{thm}
%

\begin{proof}
It suffices to consider the case $R_{\tilde{v}} = R_{\tilde{v}}^{\lambda_{\tilde{v}},\tau_{\tilde{v}},\ord}$ for each $\tilde{v} \in \tildeS_p$. 
As with \cref{thm:PDsmallRT}, we are free to replace $F$ with a finite solvable extension disjoint from the subfield of $\overline{F}$ fixed by $\rbar|_{G_{F(\zeta_p)}}$. 
After such an extension, we may assume that
\begin{itemize}
	\item every $v\in S$ splits in $F$;
	\item for every $\tilde{v} \in \tildeS$ above $p$, $\tau_{\tilde{v}} = 1$.
\end{itemize}
By \cref{thm:typeSquotient}, we may also enlarge $\tildeS$ to ensure that it contains exactly one place $w$ above any $v \in S$, setting $R_w = R_w^\square$ for each $w \in \tildeS$ with $w\nmid p$. 
We are now in the setting of \cite{ThorneAdequate} and our stated theorem is simply the combination of \cite{ThorneAdequate}*{Theorem~9.1 and Theorem~10.1}, again using \cite{Thorne2adic}*{Proposition~7.1} in place of \cite{ThorneAdequate}*{Proposition~4.4}.
\end{proof}

Finally, using the potential automorphy results of Barnet-Lamb--Gee--Geraghty--Taylor \cite{BLGGT}, we also have a potential version of the above two theorems.

\begin{thm}\label{thm:potsmallRT}
Assume $p>2$ and that every $v|p$ in $F^+$ splits in $F$. 
For each $v|p$ in $F^+$, fix a choice of place $\tilde{v}$ of $F$ above $v$, and set $\tildeS_p = \{\tilde{v}\}_{v|p \text{ in }F^+}$. 
Assume further:
\begin{ass}
\item $\mu$ is de~Rham.
\item\label{potsmallRT:adequate} $\rhobar|_{G_{F(\zeta_p)}}$ is absolutely irreducible and $\zeta_p \notin F$. 
Moreover, letting $d$ denote the maximal dimension of an irreducible subrepresentation of the restriction of $\rhobar$ to the closed subgroup of $G_{F}$ generated by all Sylow pro-$p$ subgroups, we assume $p>2(d+1)$.
\end{ass}
For each $\tilde{v} \in \tildeS_p$, fix $\lambda_{\tilde{v}} \in (\Z_+^n)^{\Hom(F_{\tilde{v}},\Qbar_p)}$ and an inertial type $\tau_{\tilde{v}}$ defined over $E$, and 
we choose $R_{\tilde{v}}$ such that one of the following hold:
\begin{casez}
	\item\label{potsmallRT:PD} For each $\tilde{v} \in \tildeS_p$, $R_{\tilde{v}}$ is a quotient of $R_{\tilde{v}}^{\lambda_{\tilde{v}},\tau_{\tilde{v}},\cris}$ corresponding to a union of potentially diagonalizable irreducible components of $\Spec R_{\tilde{v}}^{\lambda_{\tilde{v}},\tau_{\tilde{v}},\cris}$.
	\item\label{potsmallRT:ord} For each $\tilde{v} \in \tildeS_p$, $R_{\tilde{v}}$ is a quotient of $R_{\tilde{v}}^{\lambda_{\tilde{v}},\tau_{\tilde{v}},\ord}$ corresponding to a union of irreducible components of $\Spec R_{\tilde{v}}^{\lambda_{\tilde{v}},\tau_{\tilde{v}},\ord}$.
\end{casez}
Let $\calS$ be the global $\calG_n$-deformation datum 
	\[ \calS = (F/F^+, S, \tildeS_p, \calO, \rbar, \mu, \{R_{\tilde{v}}\}_{\tilde{v}\in \tildeS_p}). \]
Then the following hold:
\begin{enumerate}
\item\label{potPDsmallRT:finite} The universal type $\calS$ deformation ring $R_{\calS}$ is finite over $\calO$.
\item\label{potPDsmallRT:autpt} Given any finite extension $F^{(\mathrm{avoid})}/F$, there is a finite extension of CM fields $L/F$, disjoint from $F^{(\mathrm{avoid})}$, such that every $x\in \Spec R_{\calS}(\Qbar_p)$ is $L$-potentially automorphic.
\end{enumerate} 
\end{thm}

\begin{proof}
We can and do assume that $F^{(\mathrm{avoid})}$ contains the fixed field of $\rbar|_{G_{F(\zeta_p)}}$. 
We first take a CM extension $M/F$, disjoint from $F^{(\mathrm{avoid})}$, such that for any finite place $w$ of $M$ lying over a place in $S$, $\rhobar|_{G_w}$ is trivial.
In particular, for any finite place $w$ of $M$ lying above a place in $S$, $\rhobar|_{G_w}$ admits a characteristic zero lift $\rho_w$, which we may assume satisfy $\rho_{cw}^c \cong \rho_w^\vee \otimes \mu|_{G_w}$. 
Then, using \cref{potsmallRT:adequate}, we can apply \cite{BLGGT}*{Proposition~3.3.1} to deduce the existence of a finite extension $L/M$ of CM fields, with maximal totally real subfield $L^+$, and a regular algebraic polarized cuspidal automorphic representation $(\pi,\chi)$ of $\GL_n(\A_L)$ such that
\begin{itemize}
	\item $L$ is disjoint from $F^{(\mathrm{avoid})}$;
	\item $\rhobar|_{G_L}\otimes\Fbar_p \cong \rhobar_{\pi,\iota}$ and $\mubar|_{G_{L^+}} = \chi_\iota\epsilon^{1-n}$;
	\item $\pi$ is unramified at $p$ and outside of $S$;
	\item $\pi$ is $\iota$-ordinary.
\end{itemize}
Let $S_{L^+}$ be the set of places of $L^+$ above $S$, $\tildeS_L$ be the set of places in $L$ above $\tildeS$, and $\tildeS_{L,p}$ be the set places in $\tildeS_L$ dividing $p$. 
For each $\tilde{w} \in \tildeS_{L,p}$, let $\lambda_{\tilde{w}} \in (\Z_+^n)^{\Hom(L_{\tilde{w}},\Qbar_p)}$ be given by $\lambda_{\tilde{w},\sigma'} = \lambda_{\tilde{v},\sigma}$ if $\sigma' : L_{\tilde{w}} \hookrightarrow \Qbar_p$ extends $\sigma : F_{\tilde{v}} \hookrightarrow \Qbar_p$, and let $\tau_{\tilde{w}} = \tau_{\tilde{v}}|_{I_{\tilde{w}}}$. 
We then consider the global $\calG_n$-deformation datum
	\[\calS_L = (L/L^+, S_{L^+}, \tildeS_L, \calO, \rbar|_{G_{L^+}}, \mu|_{G_{L^+}}, \{R_{\tilde{w}}\}_{\tilde{w}\in \tildeS_{L,p}}),\]
where
\begin{itemize}
	\item if we are in \cref{potsmallRT:PD}, then for each $\tilde{w} \in \tilde{S}_{L,p}$, $R_{\tilde{w}}$ is the quotient of $R_{\tilde{w}}^{\lambda_{\tilde{w}},\tau_{\tilde{w}},\cris}$ corresponding the union of all potentially diagonalizable irreducible components;
	\item if we are in \cref{potsmallRT:ord}, then for each $\tilde{w} \in \tilde{S}_{L,p}$, $R_{\tilde{w}} = R_{\tilde{w}}^{\lambda_{\tilde{w}},\tau_{\tilde{w}},\ord}$.
\end{itemize}
There is then a canonical $\CNL_{\calO}$-algebra map
	\[ R_{\calS_L} \lra R_{\calS}, \]
which is finite by \cite{BLGGT}*{Lemma~1.2.3}. 
The theorem now follows from \cref{thm:PDsmallRT} if we are in \cref{potsmallRT:PD}, and from \cref{thm:ordsmallRT} if we are in \cref{potsmallRT:ord}.
\end{proof}

\subsection{The main theorems in the CM case}\label{sec:mainCM}
We first prove our main theorem in the potentially diagonalizable case.

\begin{thm}\label{thm:mainPD}
Assume that $p\nmid 2n$ and that every $v|p$ in $F^+$ splits in $F$. 
For each $v|p$ in $F^+$, fix a choice of place $\tilde{v}$ of $F$ above $v$, and set $\tildeS_p = \{\tilde{v}\}_{v|p \text{ in }F^+}$. 
Assume further:
\begin{ass}
\item $\mu$ is de~Rham.
\item\label{mainPD:aut} $\rhobar \otimes \Fbar_p \cong \rhobar_{\pi,\iota}$ and $\mubar = \chi_\iota\epsilon^{1-n} \mod {\frakm_{\Qbar_p}}$, where $(\pi,\chi)$ is a regular algebraic polarized cuspidal automorphic representation of $\GL_n(\A_F)$ such that $\rho_{\pi,\iota}|_{G_{\tilde{v}}}$ is potentially diagonalizable for each $\tilde{v} \in \tildeS_p$.
\item $\rhobar|_{G_{F(\zeta_p)}}$ is adequate and $\zeta_p \notin F$.
\item\label{mainPD:smooth} $H^0(G_{\tilde{v}},\ad(\rhobar)(1)) = 0$ for every $\tilde{v}\in \tildeS_p$.
\end{ass}
Then any irreducible component of $\Spec R^{\pol}$ contains an automorphic point $x$ of level potentially prime to $p$.

Moreover, assume that for every $\tilde{v} \in \tildeS_p$, we are given $\lambda_{\tilde{v}} \in (\Z_+^n)^{\Hom(F_{\tilde{v}},\Qbar_p)}$, an inertial type $\tau_{\tilde{v}}$ defined over $E$, and a nonzero potentially diagonalizable irreducible component $\calC_{\tilde{v}}$ of $\Spec R_{\tilde{v}}^{\lambda_{\tilde{v}},\tau_{\tilde{v}},\cris}$. 
Then we may assume 
the $\Qbar_p$-point of $\Spec R_{\tilde{v}}^\square$ determined by $\rho_x|_{G_{\tilde{v}}}$
lies in $\calC_{\tilde{v}}$ for each $\tilde{v}\in\tildeS_p$. 
\end{thm}

\begin{proof}
We first note that our \cref{mainPD:aut} implies that for each $\tilde{v}\in \tildeS_p$, there is a choice of $\lambda_{\tilde{v}} \in (\Z_+^n)^{\Hom(F_{\tilde{v}},\Qbar_p)}$, 
inertial type $\tau_{\tilde{v}}$ defined over $E$, and nonzero potentially diagonalizable irreducible component $\calC_{\tilde{v}}$ of $\Spec R_{\tilde{v}}^{\lambda_{\tilde{v}},\tau_{\tilde{v}},\cris}$
(enlarging $E$ if necessary using \cref{thm:coefchange}). 
We fix such a choice for each $\tilde{v} \in \tildeS_p$.  

Fix an irreducible component $\calC$ of $\Spec R^{\pol}$.
Set $R^{\loc} = \widehat{\otimes}_{\tilde{v}\in \tildeS_p} R_v^\square$, and let 
	\[ X^{\loc} = \Spec(\widehat{\otimes} R_{\tilde{v}}) \subset \Spec R^{\loc},\]
where $R_{\tilde{v}}$ is the quotient of $R_{\tilde{v}}^{\lambda_{\tilde{v}},\tau_{\tilde{v}},\cris}$ corresponding to $\calC_{\tilde{v}}$.
Choosing a lift $r : G_{F^+,S} \rightarrow \calG_n(R^{\pol})$ in the class of the universal $\mu$-polarized deformation gives a local $\CNL_{\calO}$-algebra morphism $R^{\loc} \rightarrow R^{\pol}$, and we 
let $X = X^{\loc} \times_{\Spec R^{\loc}} \Spec R^{\pol}$ under this map. 
Then $X = \Spec R_{\calS}$, where $\calS$ is the global $\calG_n$-deformation datum
	\[ \calS = (F/F^+, S, \tildeS_p, \calO, \rbar, \mu, \{R_{\tilde{v}}\}_{\tilde{v}\in\tildeS_p}). \]
\Cref{PDsmallRT:finite} of \cref{thm:PDsmallRT} implies that $X$ is finite over $\calO$. 
We also have
	\begin{itemize}
	\item $R^{\loc}$ is isomorphic to a power series over $\calO$ in $n^2\lvert \tildeS_p \rvert + n^2[F^+:\Q]$-variables by \cref{mainPD:smooth} and \cref{thm:localsmooth};
	\item $\dim X^{\loc} = 1 + \sum_{\tilde{v} \in \tildeS_p} n^2 + \frac{n(n-1)}{2}[F_{\tilde{v}}:\Q_p] = 1+n^2\lvert \tildeS_p\rvert + \frac{n(n-1)}{2}[F^+:\Q]$ by \cref{thm:crdefring};
	\item $\dim \calC \ge 1 + \frac{n(n+1)}{2}[F^+:\Q]$ by \cref{thm:polringdim}.
	\end{itemize}
We can now apply \cref{thm:thelemma} to conclude that 
	\[ \calC \cap (X\otimes_{\calO} E) = \calC \cap \Spec R_{\calS}[1/p] \ne \emptyset. \]
Applying \cref{PDsmallRT:autpt} of \cref{thm:PDsmallRT} finishes the proof.
\end{proof}

Using the polarization, the condition $H^0(G_{\tilde{v}},\ad(\rhobar)(1)) = 0$ for all $\tilde{v} \in \tildeS_p$ appearing in \cref{thm:mainPD} and \cref{thm:mainord,thm:CMgeom} below is equivalent to $H^0(G_w,\ad(\rhobar)(1)) = 0$ for all $w|p$ in $F$, which is the condition in \cref{thm:dim3dense} below. 
This is also equivalent to there being no nonzero $\F[G_w]$-equivariant map $\rhobar|_{G_w} \rightarrow \rhobar|_{G_w}(1)$, which is how this condition was stated in the introduction.
%


	
Our main theorem in the ordinary case is the following.	
	
\begin{thm}\label{thm:mainord}
Assume $p>2$ and that every $v|p$ in $F^+$ splits in $F$. 
For each $v|p$ in $F^+$, fix a choice of place $\tilde{v}$ of $F$ above $v$, and set $\tildeS_p = \{\tilde{v}\}_{v|p \text{ in }F^+}$. 
Assume further:
\begin{ass}
\item $\mu$ is de~Rham.
\item\label{mainord:ord}  $\rhobar \otimes \Fbar_p \cong \rhobar_{\pi,\iota}$ and $\mubar = \chi_\iota\epsilon^{1-n} \mod {\frakm_{\Qbar_p}}$, where $(\pi,\chi)$ is an $\iota$-ordinary regular algebraic polarized cuspidal automorphic representation of $\GL_n(\A_F)$.
\item $\rhobar(G_{F(\zeta_p)})$ is adequate and $\zeta_p \notin F$.
\item $H^0(G_v,\ad(\rhobar)(1)) = 0$ for every $\tilde{v}\in \tildeS_p$.
\end{ass}
Then any irreducible component $\calC$ of $\Spec R^{\pol}$ contains an $\iota$-ordinary automorphic point $x$.

Moreover, assume that for every $\tilde{v} \in \tildeS_p$, we are given $\lambda_{\tilde{v}} \in (\Z_+^n)^{\Hom(F_{\tilde{v}},\Qbar_p)}$, an inertial type $\tau_{\tilde{v}}$ defined over $E$, and a nonzero irreducible component $\calC_{\tilde{v}}$ of 
$\Spec R_{\tilde{v}}^{\lambda_{\tilde{v}},\tau_{\tilde{v}},\ord}$. 
Then we may assume the $\Qbar_p$-point of $\Spec R_{\tilde{v}}^\square$ determined by $\rho_x|_{G_{\tilde{v}}}$  
lies in $\calC_{\tilde{v}}$ for each $\tilde{v}\in\tildeS_p$. 
\end{thm}

\begin{proof}
We first note that \cref{thm:autgalrep} and our \cref{mainord:ord} implies that for each $\tilde{v}\in \tildeS_p$, there is a choice of $\lambda_{\tilde{v}} \in (\Z_+^n)^{\Hom(F_{\tilde{v}},\Qbar_p)}$ and inertial type $\tau_{\tilde{v}}$ defined over $E$,
such that $R_{\tilde{v}}^{\lambda_{\tilde{v}},\tau_{\tilde{v}},\ord} \ne 0$ (enlarging $E$ if necessary using \cref{thm:coefchange}).

The proof is then almost identical to the proof of \cref{thm:mainPD}, taking $X^{\loc} = \Spec(\widehat{\otimes} R_{\tilde{v}})$, with $R_{\tilde{v}}$ the quotient of $R_{\tilde{v}}^{\lambda_{\tilde{v}},\tau_{\tilde{v}},\ord}$ by the minimal prime corresponding to $\calC_{\tilde{v}}$, and using \cref{thm:ordsmallRT} instead of \cref{thm:PDsmallRT} and \cref{thm:ordring} instead of \cref{thm:crdefring}.
\end{proof}

\begin{cor}\label{thm:CMgeom}
Let the assumptions be as in either \cref{thm:mainPD} or \cref{thm:mainord}. 
Then $R^{\pol}$ is an $\calO$-flat, reduced, complete intersection ring of dimension $1 + \frac{n(n+1)}{2}[F^+:\Q]$.
\end{cor}

\begin{proof}
\Cref{thm:mainPD} and \cref{thm:mainord} imply that for any minimal prime ideal $\mathfrak{q}$ of $R^{\pol}$ 
there is an automorphic point $x\in \Spec (R^{\pol}/\frakq) (\Qbar_p)$. 
In particular, this shows $R^{\pol}/\mathfrak{q}$ is $\calO$-flat, and it has dimension $1 + \frac{n(n+1)}{2}[F^+:\Q]$ by \cite{MeSmooth}*{Theorem~C}. 
So $R^{\pol}$ is equidimensional of dimension $1+\frac{n(n+1)}{2}[F^+:\Q]$. 
This together with \cref{thm:polringdim} implies that $R^{\pol}$ is a complete intersection. 
This in turn implies that $R^{\pol}$ has no embedded prime ideals, and since $p$ does not belong to any minimal prime ideal, it is not a zero divisor and $R^{\pol}$ is $\calO$-flat. 
Applying \cite{MeSmooth}*{Theorem~C} again, we see that $R^{\pol}$ is generically regular. 
Since $R^{\pol}$ is generically regular and contains no embedded prime ideals, it is reduced.
\end{proof}

Strengthening the assumption on the residual representation slightly, we  can apply potential automorphy theorems to deduce the conclusion of \cref{thm:CMgeom} without assuming residual automorphy.

\begin{thm}\label{thm:potCMgeom}
Assume that $p>2$ and that every $v|p$ in $F^+$ splits in $F$. 
Assume further:
\begin{ass}
\item $\mu$ is de~Rham.
\item $\rhobar|_{G_{F(\zeta_p)}}$ is absolutely irreducible and $\zeta_p \notin F$.
Moreover, letting $d$ denote the maximal dimension of an irreducible subrepresentation of the restriction of $\rhobar$ to the closed subgroup of $G_{F}$ generated by all Sylow pro-$p$ subgroups, we assume $p>2(d+1)$.
\item\label{potCMgeom:smooth} $H^0(G_{\tilde{v}},\ad(\rhobar)(1)) = 0$ for every $\tilde{v}\in \tildeS_p$.
\item\label{potCMgeom:loc} One of the following hold:
	\begin{casez}
	\item\label{potCMgeom:PD} for each $\tilde{v} \in \tildeS_p$, $\rhobar|_{G_{\tilde{v}}}$ admits a regular weight potentially diagonalizable lift;
	\item\label{potCMgeom:ord} for each $\tilde{v} \in \tildeS_p$, $\rhobar|_{G_{\tilde{v}}}$ admits a regular weight ordinary lift.
	\end{casez}
\end{ass}
Then the following hold.
\begin{enumerate}
\item\label{potCMgeom:pot} For any given finite extension $F^{(\mathrm{avoid})}$ of $F$, there is a finite extension $L/F$ of CM fields, disjoint from $F^{(\mathrm{avoid})}$, such that any irreducible component $\Spec R^{\pol}$ contains an $L$-potentially automorphic point. 
\item\label{potCMgeom:geom} $R^{\pol}$ is an $\calO$-flat, reduced, complete intersection ring of dimension $1 + \frac{n(n+1)}{2}[F^+:\Q]$.
\end{enumerate}
\end{thm}

\begin{proof}
By \cref{potCMgeom:loc}, for each $\tilde{v} \in \tilde{S}_p$, there is a choice of $\lambda \in (\Z_+^n)^{\Hom(F_{\tilde{v}},\Qbar_p)}$ and an inertial type $\tau_{\tilde{v}}$ defined over $E$ (extending $E$ if necessary), such that
	\begin{itemize}
	\item $R_{\tilde{v}}^{\lambda_{\tilde{v}},\tau_{\tilde{v}},\cris}$ has a potentially diagonalizable point if we are in \cref{potCMgeom:PD}, in which case we fix a potentially diagonalizable irreducible component $\calC_{\tilde{v}}$ of $\Spec R_{\tilde{v}}^{\lambda_{\tilde{v}},\tau_{\tilde{v}},\cris}$;
	\item $R_{\tilde{v}}^{\lambda_{\tilde{v}},\tau_{\tilde{v}},\ord} \ne 0$ if we are in \cref{potCMgeom:ord}, in which case we fix an irreducible component $\calC_{\tilde{v}}$ of $\Spec R_{\tilde{v}}^{\lambda_{\tilde{v}},\tau_{\tilde{v}},\ord}$.
	\end{itemize}
The proof of \cref{potCMgeom:pot} is then exactly as in \cref{thm:mainPD}, using \cref{thm:potsmallRT} instead of \cref{thm:PDsmallRT}, and the proof of \cref{potCMgeom:geom} is exactly as in \cref{thm:potCMgeom}, using \cref{potCMgeom:pot} instead of \cref{thm:mainPD,thm:mainord} (note \cite{MeSmooth}*{Theorem~C} only requires potential automorphy).
\end{proof}

\begin{rmk}\label{rmk:PDlift}
It is expected that \cref{potCMgeom:PD} of \cref{potCMgeom:loc} in \cref{thm:potCMgeom} always holds \cite{EmertonGeeBM}*{Conjecture~A.3}, and it is known in many cases by work of Gee--Herzig--Liu--Savitt \cite{GHLS}. 
For example, assume that there is a $G_{\tilde{v}}$-stable filtration $0 = U_0 \subset U_1\subset \cdots \subset U_k = \F^n$ whose graded pieces $U_i/U_{i-1}$ are irreducible, such that there is no nonzero $\F[G_{\tilde{v}}]$-morphism $U_{i-1}(-1) \rightarrow U_i/U_{i-1}$ for any $1\le i \le k$.  Then \cite{GHLS}*{Corollary~2.1.11} implies that $\rho_{\tilde{v}}$ admits a potentially diagonalizable lift of regular weight (see \cite{GHLS}*{Examples~2.1.4}).
\end{rmk}

\begin{rmk}\label{rmk:minimal}
We note that it is possible to prove versions of the main theorems here without the potentially diagonalizable or ordinary hypothesis at the expense of a stronger assumption on the residual image. 
In particular, that each $v\in S$ splits in $F$ and that $H^0(G_w,\ad(\rhobar)(1)) = 0$ for \emph{all} places $w$ above a place in $S$, as opposed to just the ones above $p$.
This is because the only general $R = \mathbb{T}$ theorem at our disposal, without a potentially diagonalizable or ordinary assumption, is the minimal $R = \mathbb{T}$ theorem \cite{ThorneAdequate}*{Theorem~7.1}. 
To apply \cref{thm:thelemma} in this situation, it would be necessary to include all places in $S$ when defining $R^{\loc}$ and $X^{\loc}$. 
This would then force us to require that the unrestricted local deformation rings are regular at all places in $S$.
\end{rmk}

\subsection{Density of automorphic points}\label{sec:dim3dense}
We now combine our main theorems with \cite{MeSmooth} and the work of Chenevier \cite{ChenevierFern} to prove new cases of Chenevier's conjecture \cite{ChenevierFern}*{Conjecture~1.15}. 


Recall $F$ is a CM field with maximal totally real subfield $F^+$, and we have a fixed isomorphism $\Qbar_p \xrightarrow{\sim}\C$. 
We now assume that our finite set of finite places $S$ of $F^+$ contains all finite places that ramify in $F$ (as well as all those above $p$).
We restrict ourselves to dimension $3$, i.e.
	\[ \rhobar : G_{F,S} \lra \GL_3(\F) \]
is continuous and absolutely irreducible. 
We also restrict ourselves to the conjugate self dual case, i.e. we assume
	\[ \rhobar^c \cong \rhobar^\vee \otimes \mubar \quad \text{with} \quad \mubar = \epsilon^{-2}\delta_{F/F^+} \bmod {\frakm_{\calO}}.\]



\begin{thm}\label{thm:dim3dense}
Let the assumptions and notation be as above \S\ref{sec:dim3dense}.
Let $R^{\pol}$ be the universal $\epsilon^{-2}\delta_{F/F^+}$-polarized deformation ring for $\rhobar$, and let $\mathfrak{X}$ be its rigid analytic generic fibre.
Assume further:
\begin{ass}
\item $p > 2$ and is totally split in $F$.
\item $\rhobar\otimes \Fbar_p \cong \rhobar_{\pi,\iota}$, where $\pi$ is a regular algebraic conjugate self dual cuspidal automorphic representation of $\GL_3(\A_F)$ such that for each $w|p$ in $F$, $\pi_w$ is  unramified and $\rho_{\pi,\iota}|_{G_w}$ is potentially diagonalizable. 
If $p = 3$, then we further assume that $\pi$ is $\iota$-ordinary.
\item $\rhobar(G_{F(\zeta_p)})$ is adequate and $\zeta_p \notin F$.
\item $H^0(G_w,\ad(\rhobar)(1)) = 0$ for every $w|p$ in $F$.
\end{ass}
Then the set of automorphic points of level prime to $p$ in $\mathfrak{X}$ is Zariski dense.
\end{thm}

\begin{proof}
By \cite{ChenevierFern}*{Theorem~A}, the Zariski closure in $\mathfrak{X}$ of the set of automorphic points of level prime to $p$ 
has dimension at least $6[F^+:\Q]$. 
(Chenevier actually works with the universal $\delta_{F/F^+}$-polarized deformation ring,
but this is simply because of a difference in normalization: 
in \cite{ChenevierFern}, the $\rho_{\pi,\iota}$ are normalized so that $\rho_{\pi,\iota}^c \cong \rho_{\pi,\iota}^\vee$, whereas our normalization yields $\rho_{\pi,\iota}^c \cong \rho_{\pi,\iota}^\vee \otimes\epsilon^{-2}$.)
By \cref{thm:CMgeom}, $R^{\pol}$ is $\calO$-flat, reduced, and equidimensional of dimension $1+6[F^+:\Q]$. 
\Cref{thm:mainPD,thm:mainord} imply that every irreducible component of $\Spec R^{\pol}$ contains an automorphic point of level prime to $p$, and \cite{MeSmooth}*{Theorem~C} implies that $(R^{\pol})_x^\wedge$ is formally smooth over $E$ for any such point $x$. 
Applying \cref{thm:genfiblem} finishes the proof.
\end{proof}

\Cref{thm:dim3dense} and \cref{thm:speczardense} immediately imply:

\begin{cor}\label{thm:specdim3dense}
Let the assumptions and notation be as in \cref{thm:dim3dense}.
Then the set of automorphic points of level prime to $p$ in $\Spec R^{\pol}$ is Zariski dense.
\end{cor}

\section{The Hilbert modular case}\label{sec:Hilb}

We now investigate the Hilbert modular case, i.e. the case of two dimensional representations of the absolute Galois group of a totally real field. 
We first fix some assumptions and notation that will be used throughout this section.

\subsection{Setup}\label{sec:Hilbsetup}
Throughout this section we assume $p>2$ and fix an isomorphism $\iota : \Qbar_p \xrightarrow{\sim} \C$. 
We assume that our number field $F$ is totally real, and we fix a finite set of finite places $S$ of $F$ containing all places above $p$.
Let $F_S$ be the maximal extension of $F$ unramified outside of the places in $S$ and the infinite places, and we set $G_{F,S} = \Gal(F_S/F)$. 
We fix a continuous absolutely irreducible
	\[ \rhobar : G_{F,S} \lra \GL_2(\F)\]
such that $\det\rhobar$ is totally odd.
We also fix a continuous character character $\mu : G_{F,S} \rightarrow \calO^\times$ such that $\det\rhobar = \mubar$, where $\mubar$ denotes the reduction of $\mu$ modulo $\frakm_{\calO}$.
We let $R^{\univ}$ be the universal deformation ring for $\rhobar$, and let $R^\mu$ be the universal determinant $\mu$ deformation ring for $\rhobar$.

\begin{defn}
Let $R$ be a quotient of $R^{\univ}$, let $x \in \Spec R(\Qbar_p)$, and let $\rho_x$ be the pushforward of the universal deformation via $R^{\univ} \twoheadrightarrow R \xrightarrow{x} \Qbar_p$.

We say $x$ is an \emph{automorphic point} if there is a regular algebraic cuspidal automorphic representation $\pi$ of $\GL_2(\A_F)$ such that $\rho_x \cong \rho_{\pi,\iota}$.
We say 
$x$ has \emph{level prime to} $p$ if $\pi_v$ is unramified for each $v|p$ in $F$. 

Given a finite extension $L/F$ of totally real fields, we say $x$ is an $L$-\emph{potentially automorphic point} if there is a regular algebraic cuspidal automorphic representation $\pi$ of $\GL_2(\A_L)$ such that $\rho_x|_{G_L} \cong \rho_{\pi,\iota}$

We say $x$ is an \emph{essentially automorphic point} if there is a regular algebraic cuspidal automorphic representation $\pi$ of $\GL_2(\A_F)$ and a continuous character $\psi : G_{F,S} \rightarrow \Qbar_p^\times$, such $\rho_x \cong \rho_{\pi,\iota}\otimes \psi$. 
If $\pi_v$ is unramified for each $v|p$, then we say $x$ has \emph{level essentially prime to} $p$.

If $X^{\rig}$ is the rigid analytic generic fibre of $\Spf R$, and $x^{\rig} \in X^{\rig}$ is the point corresponding to $\ker(x) \subset R[1/p]$, then we say $x^{\rig}$ is an \emph{automorhpic point}, resp. an \emph{essentially automorphic point}, if $x$ is, and if this is the case we further say $x^{\rig}$ has \emph{level essentially prime to} $p$ if $x$ does.
\end{defn}

It is necessary to introduce the notion of essentially automorphic points to avoid assuming Leopoldt's conjecture. 
We will find use for the following standard lemma, cf. \cite{BockleGenFibre}*{Proposition~2.1}.

\begin{lem}\label{thm:det}
Let $\Gamma$ be the maximal pro-$p$ abelian quotient of $G_{F,S}$, and let $\Psi : G_{F,S} \rightarrow \calO[[\Gamma]]^\times$ be the tautological character. 
Let $\widetilde{\det\rhobar} : G_{F,S} \rightarrow \calO^\times$ be the Teichm\"{u}ller lift of $\det\rhobar$, and let $\widehat{\mu}^{\frac{1}{2}}: G_{F,S} \rightarrow 1+\frakm_{\calO}$ be the unique character such that $(\widehat{\mu}^{\frac{1}{2}})^2 = \mu(\widetilde{\det\rhobar})^{-1}$ (here we use that $p>2$).

The $\CNL_{\calO}$-algebra morphism $R^{\univ} \rightarrow R^\mu \widehat{\otimes}\, \calO[[\Gamma]]$ induced by $\rho^\mu \otimes\widehat{\mu}^{\frac{1}{2}}\Psi$ is an isomorphism. 
\end{lem}

This has the following immediate corollary that will allow us to deduce the existence of automorphic points in the irreducible components of the nonfixed determinant deformation ring from the existence of automorphic points in the irreducible components of fixed determinant deformation rings. 

\begin{lem}\label{thm:comps} 
Let $\calC$ be an irreducible component of $\Spec R^{\univ}$. 
After possibly enlarging $\calO$, there is a finite $p$-power order character $\theta: G_{F,S} \rightarrow \calO^\times$ such that some irreducible component of $\Spec R^{\theta\mu} \subseteq \Spec R^{\univ}$ is contained in $\calC$, where $R^{\theta\mu}$ is the universal determinant $\theta\mu$ deformation ring for $\rhobar$.
\end{lem}


\begin{prop}\label{thm:dimdet}
There is a presentation 
	\[ R^\mu \cong \calO[[x_1,\ldots,x_g]]/(f_1,\ldots,f_k)\]
with $g-k \ge 2[F:\Q]$. 
In particular, each irreducible component of $\Spec R^\mu$ has dimension at least $1+2[F:\Q]$.
\end{prop}

\begin{proof}
This follows from \cref{thm:gendetpres} and the global Euler--Poincar\'{e} characteristic formula, since $\det\rhobar$ is totally odd.
\end{proof}

\begin{cor}\label{thm:dimnodet}
There is a presentation
	\[ R^{\univ} \cong \calO[[x_1,\ldots,x_g]]/(f_1,\ldots,f_k)\]
with $g-k \ge 1+d_F+2[F:\Q]$, where $d_F$ is the Leopoldt defect for $F$ and $p$. 
In particular, every irreducible component of $\Spec R^{\univ}$ has dimension at least $2+d_F + 2[F:\Q]$.
\end{cor}

\begin{proof} 
Let $\Gamma$ be the maximal pro-$p$ abelian quotient of $G_{F,S}$. 
Then $\Gamma \cong \Z_p^{1+d_F} \times \Gamma_{\tor}$ with $\Gamma_{\tor}$ a product of finite cyclic groups of $p$-power order. 
From this it follows that $\calO[[\Gamma]]$ has a presentation
	\[ \calO[[\Gamma]] \cong \calO[[y_1,\ldots,y_r]]/(g_1,\ldots,g_s)\]
with $r-s = 1+d_F$. 
The corollary then follows from \cref{thm:dimdet} and \cref{thm:det}.
\end{proof}

We will need a smoothness result analogous to \cite{MeSmooth}*{Theorem~C} in the Hilbert modular case.

\begin{thm}\label{thm:smoothHilb}
Let $L/F$ be a finite extension of totally real fields, and let $x \in \Spec R^\mu(\Qbar_p) \subset \Spec R^{\univ}(\Qbar_p)$ be an $L$-potentially automorphic point.
Let $(R^\mu)_x^\wedge$ and $(R^{\univ})_x^\wedge$ denote the respective localizations and completions at $x$. 

If $\rhobar(G_{L(\zeta_p)})$ is adequate, then $(R^\mu)_x^\wedge$ and $(R^{\univ})_x^\wedge$ are formally smooth over $E$ of dimensions $2[F:\Q]$ and $1+d_F+2[F:\Q]$, respectively.
\end{thm}

\begin{proof}
Let $\Gamma$ be the maximal pro-$p$ abelian quotient of $\Gamma$, and fix a splitting $\Gamma \cong \Gamma_{\free}\times\Gamma_{\tor}$ with $\Gamma_{\free} \cong \Z_p^{1+d_F}$ and $\Gamma_{\tor}$ finite. 
Using \cref{thm:coefchange}, we can assume that $E$ contains the values of any $\Qbar_p$-character of $\Gamma_{\tor}$ as well as the values of $\det\rho_x$. 
Using \cref{thm:det} and choosing appropriate topological generators for $\Gamma_{\free}$, we have $R^{\univ} \cong R^\mu[[y_1,\ldots,y_{1+d_F}]][\Gamma_{\tor}]$, and we can assume $y_1,\ldots,y_{1+d_F}$ lie in the kernel of $x$. 
This implies $(R^{\univ})_x^\wedge \cong (R^\mu)_x^\wedge[[y_1,\ldots,y_{1+d_F}]]$, and the result for $R^{\univ}$ follows from that for $R^\mu$. 

The result for $R^\mu$ follows from \cite{MeSmooth}*{Theorem~B} using an argument as in \cite{KisinGeoDefs}*{\S8}. 
We give a sketch. 
Let $k = R^\mu[1/p]/\ker(x)$, and let $\rho : G_{F,S} \rightarrow \GL_2(k)$ be the pushforward of the universal $R^\mu$-valued deformation via $R^\mu[1/p] \twoheadrightarrow k$. 
We let $\ad^0(\rho)$ denote the trace zero subspace of the Lie algebra of $\GL_2(k)$, equipped with the adjoint $G_{F,S}$-action $\ad\circ \rho$. 
By \cite{MeSmooth}*{Theorem~B}, the geometric Bloch--Kato Selmer group
	\[ H_g^1(G_{F,S},\ad^0(\rho)) := \ker\big( H^1(G_{F,S},\ad^0(\rho)) \rightarrow \prod_{v|p} 
	H^1(G_v,B_{\dR}\otimes_{\Q_p}\ad^0(\rho))\big)\]
is trivial. 
Using this, the argument of \cite{KisinGeoDefs}*{Theorem~8.2} shows that $H^1(G_{F,S},\ad^0(\rho))$ injects into 
	\[ \prod_{v|p} \mathrm{Fil}^0 (B_{\dR}\otimes_{\Q_p}\ad^0(\rho))^{G_v}, \]
and this space has $k$-dimension $\sum_{v|p} 2[F_v:\Q_p] = 2[F:\Q]$, since the Hodge--Tate weights for $\rho|_{G_v}$ are distinct for each $v|p$ and each embedding $\sigma : F_v \hookrightarrow \Qbar_p$. 
On the other hand, using the argument of \cite{KisinOverConvFM}*{Proposition~9.5},  $H^1(G_{F,S},\ad^0(\rho))$ is isomorphic to the tangent space of $(R^\mu)_x^\wedge$ and $\dim(R^\mu)_x^\wedge \ge 2[F:\Q]$ by \cref{thm:dimdet}.
\end{proof}

\subsection{Small $R = \mathbb{T}$ theorems from the literature}\label{sec:smallRTHilb}
Let $\tildeS \subseteq S$, and for each $v\in \tildeS$, let $R_v$ be a quotient of $R_v^\square : = R_{\rhobar|_{G_v}}^\square$ that represents a local deformation problem.
By \cref{thm:defquolem}, the rings $R_v$ in \cref{thm:smallHilb,thm:potsmallHilb} below represent local deformation problems.

We refer to the tuple
	\[ \calS = (F, S, \tildeS, \calO, \rhobar,\mu, \{R_v\}_{v\in \tildeS}) \]
as a \emph{global} $\GL_2$-\emph{deformation datum}. 
A \emph{type} $\calS$ deformation of $\rhobar$ is a deformation $\rho : G_{F,S} \rightarrow \GL_2(A)$ with $A$ a $\CNL_{\calO}$-algebra such that for any (equivalently for one) lift $\rho$ in its deformation class 
	\begin{itemize}
	\item $\det\rho = \mu$, and 
	\item for each $v\in \tildeS$, the $\CNL_{\calO}$-morphism $R_v^\square \rightarrow A$ induced by the lift $\rho|_{G_v}$ of $\rhobar|_{G_v}$, factors through $R_v$.
\end{itemize}
The set valued functor $D_{\calS}$ on $\CNL_{\calO}$ that takes a $\CNL_{\calO}$-algebra $A$ to the set of deformations of type $\calS$ is representable.
Indeed a choice of lift $\rho^\mu$ in the universal $R^\mu$-valued deformation determines a $\CNL_{\calO}$-algebra morphism $\widehat{\otimes}_{v\in\tildeS} R_v^\square \rightarrow R^\mu$, which is canonical up to automorphisms of $\widehat{\otimes}_{v\in\tildeS} R_v^\square$. 
Then $D_{\calS}$ is represented by
	\[ R_{\calS} := R^\mu \otimes_{(\widehat{\otimes}_{v\in\tildeS} R_v^\square)} (\widehat{\otimes}_{v\in\tildeS} R_v).\]
We call $R_{\calS}$ the \emph{universal type} $\calS$ \emph{deformation ring}. 

We will deduce the following small $R = \mathbb{T}$ theorem from \cref{thm:PDsmallRT}, using an input due to Barnet-Lamb--Gee--Geraghty \cite{BLGGOrdLifts2}.
As in \S\ref{sec:CM}, both \cref{smallHilb:finite,smallHilb:modular} are crucial for our main theorems in this section.

\begin{thm}\label{thm:smallHilb}
Recall we have assumed $p>2$ and $\det\rhobar$ is totally odd. 
Assume further:
	\begin{ass}
	\item $\mu$ is de~Rham
	\item $\rhobar \otimes\Fbar_p \cong \rhobar_{\pi,\iota}$, where $\pi$ is regular algebraic cuspidal automorphic representation of $\GL_2(\A_F)$.
	\item\label{smallHilb:5} $\rhobar(G_{F(\zeta_p)})$ is adequate.
	\end{ass}	
For each $v|p$, let $\lambda_v \in (\Z_+^n)^{\Hom(F_v,\Qbar_p)}$ and $\tau_v$ be an inertial type defined over $E$, and let 
$R_v$ be a quotient $R_v^{\lambda_v,\tau_v,\cris,\mu}$ corresponding to a union of potentially diagonalizable irreducible components. 
Let $\calS$ be the global $\GL_2$-deformation datum
	\[ \calS = (F,S,S_p,\calO,\rhobar,\mu,\{R_v\}_{v\in S_p}).\]
and let $R_{\calS}$ be the universal type $\calS$ deformation ring.
Then	
	\begin{enumerate}
	\item\label{smallHilb:finite} $R_{\calS}$ is finite over $\calO$.
	\item\label{smallHilb:modular} Every $\Qbar_p$-point of $\Spec R_{\calS}$ is an automorphic point of level potentially prime to $p$.
	\end{enumerate}
\end{thm}

\begin{proof}
By \cite{BLGGOrdLifts2}*{Theorem~2.1.2} (and it's proof), there is a finite solvable totally real extension $L/F$, disjoint from the fixed field of $\rhobar|_{G_{F(\zeta_p)}}$, and a regular algebraic cuspidal automorphic representation $\pi'$ of $\GL_2(\A_L)$ such that $\rhobar|_{G_L}\otimes\Fbar_p \cong \rhobar_{\pi',\iota}$ and such that $\rho_{\pi',\iota}|_{G_w}$ is potentially diagonalizable for any $w|p$ in $L$. 
Let $\chi$ denote the central character of $\pi'$, and let $S_L$ denote the set of places in $L$ above those in $S$.
 
Now choose a quadratic CM extension $M/L$, disjoint from $\rhobar|_{G_{L(\zeta_p)}}$, such that each $w|p$ in $L$ splits in $M$. 
Using the standard symplectic pairing for $\GL_2$ and \cref{thm:Gdhoms}, $\rhobar|_{G_M}$ extends to a continuous homomorphism
	\[ \rbar : G_{L,S_L} \lra \calG_2(\F) \]
with $\nu\circ \rbar = \mubar|_{G_L}$.
Using \cite{BLGGT}*{Lemma~2.2.2}, there is a regular algebraic cuspidal automorphic representation $\Pi$ of $\GL_2(\A_M)$ such that $(\Pi,\chi)$ is polarized and such that $\rho_{\pi',\iota}|_{G_M} = \rho_{\Pi,\iota}$. 
In particular, $\rbar|_{G_M}\otimes \Fbar_p \cong \rhobar_{\Pi,\iota}$, $\mubar|_{G_L} = \chi_\iota\epsilon^{-1} \mod {\frakm_{\Qbar_p}}$, and $\rho_{\Pi,\iota}|_{G_w}$ is potentially diagonalizable for each $w|p$ in $M$. 
Note also that $\rhobar(G_{M(\zeta_p)})$ is adequate and $\zeta_p \notin M$, by choice of $M$.
 
For each $w|p$ in $L$, fix a choice of $\tilde{w}$ above $w$ in $M$, and set $\tildeS_p = \{\tilde{w}\}_{w|p \text{ in }L}$.
For each $\tilde{w} \in \tildeS_p$, if $v$ is the place below it in $F$, we let $\tau_{\tilde{w}} = \tau_v|_{I_{\tilde{w}}}$ and $\lambda_{\tilde{w}} \in (\Z_+^n)^{\Hom(M_{\tilde{w}},\Qbar_p)}$ be given by $\lambda_{\tilde{w},\sigma'} = \lambda_{v,\sigma}$ if $\sigma' : M_{\tilde{w}} \hookrightarrow \Qbar_p$ extends $\sigma : F_v \hookrightarrow \Qbar_p$. 
For each $\tilde{w} \in \tildeS_p$, we then let $R_{\tilde{w}}$ be the quotient of $R_{\tilde{w}}^{\lambda_{\tilde{w}},\tau_{\tilde{w}},\cris}$ corresponding to the union of all potentially diagonalizable irreducible components.
We then have a global $\calG_2$-deformation datum
	\[ \calS_M = (M/L, S_L, \tilde{S}_p, \calO, \rbar, \mu, \{R_{\tilde{w}}\}_{\tilde{w} \in \tildeS_p}),\]
and we let $R_{\calS_M}$ be the universal type $\calS_M$ deformation ring. 
For any type $\calS$ deformation $\rho$ of $\rhobar$ to a $\CNL_{\calO}$-algebra $A$, again using the standard symplectic pairing for $\GL_2$ and \cref{thm:Gdhoms}, we obtain a type $\calS_M$ deformation
	\[ r : G_{L,S_L} \lra \calG_2(A). \]
We deduce the existence of a $\CNL_{\calO}$-algebra map $R_{\calS_L} \rightarrow R_{\calS}$. 
A standard argument using Nakayama's Lemma and \cite{KW0}*{Lemma~3.6} shows that this map is finite. 
\Cref{smallHilb:finite} now follows from \cref{PDsmallRT:finite} of \cref{thm:PDsmallRT}, and \ref{smallHilb:modular} follows from \cref{PDsmallRT:autpt} of \cref{thm:PDsmallRT} and \cite{BLGGT}*{Lemma~2.2.2}. 
\end{proof}

We again combine this with potential automorphy theorems.

\begin{thm}\label{thm:potsmallHilb}
Let the notation and assumptions be as in the beginning of this section, and assume further:
	\begin{ass}
	\item $\mu$ is de~Rham.
	\item $\rhobar(G_{F(\zeta_p)})$ is adequate.
	\end{ass}
For each $v|p$, let $\lambda_v \in (\Z_+^n)^{\Hom(F_v,\Qbar_p)}$ and $\tau_v$ be an inertial type defined over $E$, and let 
$R_v$ be a quotient $R_v^{\lambda_v,\tau_v,\cris,\mu}$ corresponding to a union of potentially diagonalizable irreducible components. 
Let $\calS$ be the global $\GL_2$-deformation datum
	\[ \calS = (F,S,S_p,\calO,\rhobar,\mu,\{R_v\}_{v\in S_p}).\]
and let $R_{\calS}$ be the universal type $\calS$ deformation ring.
Then	
	\begin{enumerate}
	\item\label{potsmallHilb:finite} $R_{\calS}$ is finite over $\calO$.
	\item\label{potsmallHilb:modular} Given any finite extension $F^{(\mathrm{avoid})}/F$, there is a finite extension of totally real fields $L/F$, disjoint from $F^{(\mathrm{avoid})}$, such that every $\Qbar_p$-point of $\Spec R_{\calS}$ is $L$-potentially automorphic.
	\end{enumerate}
\end{thm}

\begin{proof}
We can and do assume that $F^{(\mathrm{avoid})}$ contains the fixed field of $\rhobar|_{G_{F(\zeta_p)}}$. 
By \cite{TaylorRmkFM}*{Corollary~1.7}, we can find a finite totally real extension $L/F$ and a regular algebraic cuspidal automorphic representation $\pi$ of $\GL_2(\A_L)$ such that $\rhobar|_{G_L}\otimes \Fbar_p \cong \rhobar_{\pi,\iota}$. 
Further, by using the refinement \cite{HSBT}*{Proposition~2.1} of Moret-Bailly's theorem in place of \cite{TaylorRmkFM}*{Theorem~G}, we may assume $L$ is disjoint from $F^{(\mathrm{avoid})}$ (see also \cite{BLGGT}*{Theorem~3.1.2}).
Let $S_L$, resp. $S_{L,p}$, denotes the set of places of $L$ above $S$, resp. above $p$. 
For each $w \in S_{L,p}$, if $v$ is the place below it in $F$, we let $\tau_w = \tau_v|_{I_w}$ and $\lambda_w \in (\Z_+^n)^{\Hom(L_w,\Qbar_p)}$ be given by $\lambda_{w,\sigma'} = \lambda_{v,\sigma}$ if $\sigma' : L_w \hookrightarrow \Qbar_p$ extends $\sigma : F_v \hookrightarrow \Qbar_p$. 
For each $w\in S_{L,p}$, we then let $R_w$ be the quotient of $R_w^{\lambda_w,\tau_w,\cris,\mu}$ corresponding the union of all potentially diagonalizable irreducible components. 
We have the global $\GL_2$-deformation datum 
	\[ \calS_L = (L,S_L,S_{L,p},\calO,\rhobar|_{G_L},\mu|_{G_L},\{R_w\}_{w\in S_{L,p}}),\]
and the universal type $\calS_L$-deformation ring $R_{\calS_L}$. 
There is a natural map $R_{\calS_L} \rightarrow R_{\calS}$, which is finite (\cite{KW0}*{Lemma~3.6}). 
The theorem now follows from applying \cref{thm:smallHilb} to $R_{\calS_L}$.
\end{proof}

\subsection{The main theorems in the Hilbert modular case}\label{sec:mainHilb}

We can now prove our main theorems in the Hilbert modular case. 

\begin{thm}\label{thm:mainHilbdet}
Recall we have assumed $p>2$ and that $\det\rhobar$ is totally odd.
Assume further:
	\begin{ass}
	\item $\mu$ is de~Rham.	
	\item\label{mainHilbdet:aut} $\rhobar \otimes\Fbar_p \cong \rhobar_{\pi,\iota}$, where $\pi$ is a regular algebraic cuspidal automorphic representation of $\GL_2(\A_F)$.
	\item $\rhobar(G_{F(\zeta_p)})$ is adequate.
	\item\label{mainHilbdet:smooth} $H^0(G_v,\ad^0(\rhobar)(1)) = 0$ for every $v|p$. 
	\end{ass}
Then any irreducible component $\calC$ of $\Spec R^\mu$ contains an automorphic point $x$ of level potentially prime to $p$.

Moreover, assume that for every $v|p$ in $F$, we are given $\lambda_v \in (\Z_+^n)^{\Hom(F_v,\Qbar_p)}$, an inertial type $\tau_v$ defined over $E$, and a nonzero potentially diagonalizable irreducible component $\calC_v$ of $\Spec R_v^{\lambda_v,\tau_v,\cris,\mu}$. 
Then we can assume that the $\Qbar_p$-point of $\Spec R_v^\square$ determined by $\rho_x|_{G_v}$ lies in $\calC_v$ for each $v|p$. 
\end{thm}

\begin{proof}
We first note that by \cref{thm:BTlift}, for each $v|p$ there is a choice of $\lambda_v \in (\Z_+^n)^{\Hom(F_v,\Qbar_p)}$, inertial type $\tau_v$ defined over $E$, and nonzero potentially diagonalizable irreducible component $\calC_v$ of $\Spec R_v^{\lambda_v,\tau_v,\cris,\mu}$, which we fix.
The proof now is similar to \cref{thm:mainPD}; we give the details. 

Fix an irreducible component $\calC$ of $\Spec R^\mu$. 
Set $R^{\loc} = \widehat{\otimes}_{v\in S_p} R_v^{\square,\mu}$, and let
	\[ X^{\loc} = \Spec(\widehat{\otimes}_{v\in S_p} R_v) \subseteq \Spec R^{\loc}, \]
where $R_v$ is the quotient of $R_v^{\lambda_v,\tau_v,\cris,\mu}$ corresponding to $\calC_v$. 
Choosing a lift $\rho^\mu : G_{F,S} \rightarrow \GL_2(R^\mu)$ in the class of the universal determinant $\mu$ deformation gives a local $\CNL_{\calO}$-algebra morphism $R^{\loc} \rightarrow R^\mu$, and we let $X = X^{\loc} \times_{\Spec R^{\loc}} \Spec R^\mu$.
Then $X = \Spec R_{\calS}$, where $\calS$ is the global $\GL_2$-deformation datum
	\[ \calS = (F, S, S_p, \calO, \rhobar, \mu, \{R_v\}_{v\in S_p}). \]
\Cref{smallHilb:finite} of \cref{thm:smallHilb} implies that $X$ is finite over $\calO$. 
We also have 
	\begin{itemize}
	\item $R^{\loc}$ is isomorphic to a power series over $\calO$ in $3|S_p| + 3[F:\Q]$-variables by \cref{mainHilbdet:smooth} and \cref{thm:localsmooth};
	\item $\dim X^{\loc} = 1+\sum_{v\in S_p} 3 + [F_v : \Q_p] = 1+ 3\lvert S_p \rvert + [F:\Q]$ by \cref{thm:crdefringdet};
	\item $\dim \calC \ge 1 + 2[F:\Q]$ by \cref{thm:dimdet}.
	\end{itemize}
We can now apply \cref{thm:thelemma} to conclude that 
	\[ \calC \cap (X\otimes_{\calO} E) = \calC \cap \Spec R_{\calS}[1/p] \ne \emptyset. \]
Applying \cref{smallHilb:modular} \cref{thm:smallHilb} finishes the proof.
\end{proof}

\begin{rmk}\label{rmk:comppss}
An analogue of \cref{thm:mainHilbdet} can be proved with ``potentially diagonalizable" replaced with ``semistable with distinct Hodge--Tate weights", at the cost of assuming $p$ is totally split in $F$ by using the $R=\bbT$ theorem of Kisin \cite{KisinFM} and the subsequent improvements on the Breuil--M\'{e}zard conjecture due to Pa\v{s}k\={u}nas \cite{PaskunasBM} and Hu--Tan \cite{HuTanBM}. 
We leave the precise statements to the reader.
\end{rmk}

We also have the potential version of the above theorem.

\begin{thm}\label{thm:mainpotHilb}
Recall we have assumed $p>2$ and that $\det\rhobar$ is totally odd.
Assume further:
	\begin{ass}
	\item $\mu$ is de~Rham.
	\item\label{mainpotHilb:TW} $\rhobar(G_{F(\zeta_p)})$ is adequate.
	\item\label{mainpotHilb:smooth} $H^0(G_v,\ad^0(\rhobar)(1)) = 0$ for every $v|p$. 
	\end{ass}
Then given any finite extension $F^{(\mathrm{avoid})}/F$, there is a finite extension $L/F$ of totally real fields, disjoint $F^{(\mathrm{avoid})}$, such that every irreducible component $\calC$ of $\Spec R^\mu$ contains an $L$-potentially automorphic point $x$.

Moreover, assume that for every $v|p$ in $F$, we are given $\lambda_v \in (\Z_+^n)^{\Hom(F_v,\Qbar_p)}$, an inertial type $\tau_v$ defined over $E$, and a nonzero potentially diagonalizable irreducible component $\calC_v$ of $\Spec R_v^{\lambda_v,\tau_v,\cris,\mu}$. 
Then we can assume that the $\Qbar_p$-point of $\Spec R_v^\square$ determined by $\rho_x|_{G_v}$ lies in $\calC_v$ for each $v|p$. 
\end{thm}

\begin{proof}
The proof is identical to the proof of \cref{thm:mainHilbdet}, except using \cref{thm:potsmallHilb} in place of \cref{thm:smallHilb}.
\end{proof}

\Cref{thm:comps,thm:mainHilbdet} immediately imply.

\begin{cor}\label{thm:mainHilb}
\begin{enumerate}
\item\label{mainHilb:aut} Let the assumptions be as in \cref{thm:mainHilbdet}. 
Then any irreducible component of $\Spec R^{\univ}$ contains an automorphic point.
\item\label{mainHilb:pot} Let the assumptions be as in \cref{thm:mainpotHilb}. 
Then for any given finite extension $F^{(\mathrm{avoid})}$, there is a finite extension $L/F$ of totally real fields, disjoint form $F^{(\mathrm{avoid})}$, such that every irreducible component of $\Spec R^{\univ}$ contains an $L$-potentially automorphic point.
\end{enumerate}
\end{cor}

As in \S\ref{sec:CM}, we also deduce ring theoretic properties for $R^{\univ}$ and $R^\mu$. 

\begin{cor}\label{thm:Hilbgeom}
Recall we have assumed $p>2$ and that $\det\rhobar$ is totally odd. 
Assume further:
	\begin{ass}
	\item $\rhobar(G_{F(\zeta_p)})$ is adequate. 
	\item $H^0(G_v,\ad^0(\rhobar)(1)) = 0$ for every $v|p$. 
	\end{ass}
Then $R^{\univ}$ and $R^\mu$ are $\calO$-flat, reduced, complete intersection rings of dimensions $2+d_F+2[F:\Q]$ and $1+2[F:\Q]$, respectively.
\end{cor}

\begin{proof}
The proof for $R^{\univ}$ is similar to that of \cref{thm:CMgeom}. 
We give the details.

\Cref{mainHilb:pot} of \cref{thm:mainHilb} implies that there is a finite extension $L/F$ of totally real fields, disjoint form the fixed field of $\rhobar|_{G_{F(\zeta_p)}}$, such that for any minimal prime ideal $\mathfrak{q}$ of $R^{\univ}$ there is an $L$-potentially automorphic point $x\in \Spec (R^{\univ}/\frakq) (\Qbar_p)$. 
In particular, this shows $R^{\univ}/\mathfrak{q}$ is $\calO$-flat, and it has dimension $2 + d_F + 2[F:\Q]$ by \cref{thm:smoothHilb}. 
So $R^{\univ}$ is equidimensional of dimension $2+d_F + 2[F:\Q]$. 
This together with \cref{thm:dimnodet} implies that $R^{\univ}$ is a complete intersection. 
This in turn implies that $R^{\univ}$ has no embedded prime ideals, and since $p$ does not belong to any minimal prime ideal, it is not a zero divisor and $R^{\univ}$ is $\calO$-flat. 
Applying \cref{thm:smoothHilb} again, we see that $R^{\univ}$ is generically regular. 
Since $R^{\univ}$ is generically regular and contains no embedded prime ideals, it is reduced.

When $\mu$ is de~Rham, the proof for $R^\mu$ is the same, using \cref{thm:mainpotHilb} instead of \cref{mainHilb:pot} of \cref{thm:mainHilb}, and \cref{thm:dimdet} instead of \cref{thm:dimnodet}.
When $\mu : G_{F,S} \rightarrow \calO^\times$ is an arbitrary character lifting $\det\rhobar$, we choose a character $\mu' : G_{F,S} \rightarrow \calO^\times$ lifting $\det\rhobar$ such that $\mu'$ is de~Rham. 
Since $p>2$, there is a character $(\mu'\mu^{-1})^{\frac{1}{2}} : G_{F,S} \rightarrow 1+\frakm_{\calO}$ such that $((\mu'\mu^{-1})^{\frac{1}{2}})^2 = \mu'\mu^{-1}$. 
Twisting the universal determinant $\mu$-deformation of $\rhobar$ by $(\mu'\mu^{-1})^{\frac{1}{2}}$ induces a $\CNL_{\calO}$-algebra isomorphism $R^{\mu'} \xrightarrow{\sim} R^\mu$. 
This completes the proof.
\end{proof}


Finally, combining the above with \cref{thm:smoothHilb} and the work of Gouvea--Mazur and Chenevier, we obtain:

\begin{thm}\label{thm:Hilbdense}
Assume $p>2$ and that $\det\rhobar$ is totally odd.
Let $\mathfrak{X}$ be the rigid analytic generic fibre of $\Spf R^{\univ}$.
Assume further:
	\begin{ass}
	\item $p$ is totally split in $F$, and if $F \ne \Q$, then $[F:\Q]$ is even. 
	\item\label{Hilbdense:aut} $\rhobar \otimes\Fbar_p \cong \rhobar_{\pi,\iota}$, where $\pi$ is a regular algebraic cuspidal automorphic representation of $\GL_2(\A_F)$ such that for all $v|p$, $\pi_v$ is unramified and $\rho_{\pi,\iota}|_{G_v}$ is potentially diagonalizable.
	\item\label{Hilbdense:adequate} $\rhobar(G_{F(\zeta_p)})$ is adequate.
	\item\label{Hilbdense:smooth} $H^0(G_v,\ad^0(\rhobar)(1)) = 0$ for every $v|p$. 
	\end{ass}
Then the set of essentially automorphic points of level essentially prime to $p$ in $\mathfrak{X}$ is Zariski dense.
If Leopoldt's conjecture holds for $F$ and $p$, then the set of automorphic points of level essentially prime to $p$ in $\mathfrak{X}$ is Zariski dense.
\end{thm}

\begin{proof}
Note that if Leopoldt's conjecture holds for $F$ and $p$, then the set of essentially automorphic points of level essentially prime to $p$ and the set of automorphic points of level essentially prime to $p$ coincide, so the second claim follows from the first. 

Work of Gouvea and Mazur when $F = \Q$ (see \cite{Emertonpadic}*{Corollary~2.28}) and Chenevier \cite{ChenevierFern}*{Theorem~5.9} when $[F:\Q]$ is even, implies that the Zariski closure in $\mathfrak{X}$ of the set of essentially automorphic points of level essentially prime to $p$
has dimension at least $1+d_F + 2[F:\Q]$. 
By \cref{thm:Hilbgeom}, $R^{\univ}$ is $\calO$-flat, reduced, and equidimensional of dimension $2+d_F + 2[F:\Q]$. 
Then by \cref{thm:smoothHilb} and \cref{thm:genfiblem}, it suffices to prove that every irreducible component of $\Spec R^{\univ}$ contains an automorphic point of level essentially prime to $p$.

Fix an irreducible component $\calC$ of $\Spec R^{\univ}$, and let $\mu = \det \rho_{\pi,\iota}$. 
By \cref{thm:comps}, there is a finite $p$-power order character $\theta : G_{F,S} \rightarrow \calO^\times$ (extending $E$ if necessary) and an irreducible component $\calC_{\theta\mu}$ of $\Spec R^{\theta\mu}$ such that $\calC_{\theta\mu}\subset \calC$. 
Since $p>2$, there is a finite $p$-power order character $\eta : G_{F,S} \rightarrow \calO^\times$ such that $\eta^2 = \theta$. 
Twisting by $\eta$ gives an isomorphism $R^{\mu} \xrightarrow{\sim} R^{\theta\mu}$, so there is an irreducible component $\calC_\mu$ of $\Spec R^\mu$ such that twisting by $\eta$ yields an isomorphism $\calC_\mu \xrightarrow{\sim} \calC_{\theta\mu}$. 
By our assumption on $\pi$, for each $v|p$ there is a choice of $\lambda_v \in (\Z_+^n)^{\Hom(F_v,\Qbar_p)}$ such that $R_v^{\lambda_v,\cris} \ne 0$ and has a potentially diagonalizable irreducible component $\calC_v$. 
Applying \cref{thm:mainHilbdet}, we deduce there is an automorphic point $x$ on $\calC_\mu$ whose image in $\Spec R_v^\square$ lies on $\calC_v$ for each $v|p$. 
In particular, $\rho_x|_{G_v}$ is crystalline for each $v|p$. 
By local global compatibility, we deduce that $x$ has level prime to $p$. 
Then $\rho_x \otimes \eta$ is an automorphic point on $\calC_{\theta\mu} \subset \calC$ of level essentially prime to $p$, which completes the proof.
\end{proof}

\Cref{thm:Hilbdense} and \cref{thm:speczardense} immediately imply:

\begin{cor}\label{thm:specHilbdense}
Let the assumptions and notation be as in \cref{thm:Hilbdense}.
Then the set of automorphic points of level essentially prime to $p$ in $\Spec R^{\univ}$ is Zariski dense.
If Leopoldt's conjecture holds for $F$ and $p$, then the set of automorphic points of level essentially prime to $p$ in $\Spec R^{\univ}$ is Zariski dense.
\end{cor}

\begin{rmk}\label{rmk:primetop}
Even if Leopoldt's conjecture holds for $F$ and $p$, it is necessary to use automorphic points of level essentially prime to $p$, i.e. it is not in general true that the automorphic points of pime to $p$ level are Zariski dense in $\Spec R^{\univ}$. 
Indeed, let $\Gamma$ be the maximal pro-$p$ abelian quotient of $G_{F,S}$, and let $\Gamma_{\tor}$ denote its torsion subgroup.
Then $\calO[[\Gamma]]$ is the universal deformation ring for lifts of $\det\rhobar$, and taking determinants yields a $\CNL_{\calO}$-algebra map $\calO[[\Gamma]] \rightarrow R^{\univ}$ that takes crystalline points of $\Spec R^{\univ}$ to crystalline points of $\Spec \calO[[\Gamma]]$. 
Applying \cref{thm:det} with $\mu$ equal to the Teichm\"{u}ller lift of $\det\rhobar$, we see that $\det\rho^{\univ}$ is the universal deformation of $\det\rhobar$, from which it follows that the map $\Spec R^{\univ} \rightarrow \Spec \calO[[\Gamma]]$ is surjective. 
Hence, if the crystalline points are dense in $\Spec R^{\univ}$, they will be dense in $\Spec \calO[[\Gamma]]$, which isn't true in general. 
For example, say there are two places $v$ and $w$ above $p$ in $F$ such that $F_v$ and $F_w$ both contain a primitive $p$th root of unity. 
The natural map 
	\[\mu_p(\calO_{F_v}) \times \mu_p(\calO_{F_w}) \lra \Gamma_{\tor} \]
coming from class field theory is injective ($F$ is totally real and $p>2$), so we can find a character $\theta : \Gamma_{\tor} \rightarrow \calO^\times$ (extending $\calO$ if necessary) that is trivial on $\mu_p(\calO_{F_v})$ but not on $\mu_p(\calO_{F_w})$. 
This character $\theta$ determines an irreducible component of $\Spec \calO[[\Gamma]]$, and any point on this component yields a character whose restriction to $\Gamma_{\tor}$ equals $\theta$. 
But since $F$ is totally real, any crystalline character of $G_{F,S}$ is of the form $\phi\epsilon^m$ for some $m\in\Z$ and a finite order character $\phi$ that is unramified at all places above $p$. 
In particular a crystalline character of $G_{F,S}$ is trivial on $\mu_p(\calO_{F_v})$ if and only if it is trivial on $\mu_p(\calO_{F_w})$.
\end{rmk}

\begin{rmk}\label{rmk:overQ} 
We now compare \cref{thm:Hilbdense} with the previous results \cite{BockleDensity} and \cite{EmertonLocGlob}*{Theorem~1.2.3} for $F = \Q$. 
We will temporarily use homological normalizations (i.e. uniformizers correspond to arithmetic Frobenii and Hodge--Tate weights are normalized so that $\epsilon$ has Hodge--Tate weight $1$) for the Galois representations associated to modular forms, as this is more common in the literature we quote.

First, we claim that \cref{Hilbdense:aut}  always holds for absolutely irreducible odd $\rhobar$ when $F = \Q$.
Indeed, there is $m\in\Z$ such that $\rhobar\otimes \epsilonbar^m$ has Serre weight $k \in [2,p+1]$.
By the strong form of Serre's conjecture, 
there is cuspform $f$ of level prime to $p$ and (classical) weight $k$ such that $\rhobar_{f,\iota} \cong \rhobar \otimes \epsilonbar^m$. 
If $k\le p$, then $\rho_{f,\iota}|_{G_p}$ is potentially diagonalizable by \cite{GaoLiuFLPD}*{Theorem~3.0.3}. 
If $k = p+1$, then $\rhobar_{f,\iota}|_{I_p}$ is an extension of the trivial character by the cyclotomic character, and \cite{KW2}*{Lemma~3.5} implies that $\rho_{f,\iota}$ is ordinary crystalline, hence potentially diagonalizable by \cite{BLGGT}*{Lemma~1.4.3}. 
Applying the twist $(\abs{\cdot}_{\A}\circ\det)^m$ to the regular algebraic cuspidal automorphic representation of $\GL_2(\A_{\Q})$ generated by $f$ yields a regular algebraic cuspidal automorphic representation $\pi$ satisfying \cref{Hilbdense:aut} of \cref{thm:Hilbdense}.

Secondly, the assumption that $H^0(G_v,\ad^0(\rhobar)(1))=0$ is easily checked to be equivalent to the following (extending scalars if necessary):
\begin{itemize}
	\item $\rhobar|_{G_v} \not\cong \chibar \otimes \begin{pmatrix}
	1 & \ast \\ & \epsilonbar
	\end{pmatrix}$ for any character $\chibar : G_v \rightarrow \F^\times$, and
	\item if $[F_v(\zeta_p):F_v] = 2$, then $\rhobar$ is not isomorphic to the induction of a character $\overline{\theta} : G_{F_v(\zeta_p)} \rightarrow \F^\times$. 
	\end{itemize}
In particular, \cref{thm:Hilbdense} removes the assumption from \cite{BockleDensity} and \cite{EmertonLocGlob}*{Theorem~1.2.3} requiring the semisimplification of $\rhobar|_{G_p}$ to be nonscalar.

However there is one situation covered by \cite{BockleDensity} and \cite{EmertonLocGlob}*{Theorem~1.2.3} that we do not treat here, namely if $p=3$ and the projective image of $\rhobar|_{G_{F(\zeta_3)}}$ is conjugate to $\PSL_2(\F_3)$. 
By \cite{BLGGU2}*{Proposition~6.5}, if $p>2$ and $\rhobar|_{G_{F(\zeta_p)}}$ acts absolutely irreducibly but does not have adequate image, then either
	\begin{itemize}
	\item $p = 3$ and the projective image of $\rhobar|_{G_{F(\zeta_3)}}$ is conjugate to $\PSL_2(\F_3)$, or
	\item $p = 5$ and the projective image of $\rhobar|_{G_{F(\zeta_5)}}$ is conjugate to $\PSL_2(\F_5)$.
	\end{itemize}
The latter case does not occur when $\det\rhobar$ is totally odd and $p=5$ is unramified in $F$. 
It would be possible to treat the former case using the main ideas in this article and $\GL_2$-automorphy lifting theorems, but this would also require proving \cref{thm:smoothHilb} using $\GL_2$-patching, instead of quoting \cite{MeSmooth}*{Theorem~B}, so we do not pursue it here.
\end{rmk}
	
\subsection{An example}\label{sec:eg}
We finish by using an example due to Serre (see \cite{RibetIrred}*{\S2}) to illustrate the subtleties, in particular the failure of \cref{thm:thelemma}, that can occur when $R^{\loc}$ is not regular. 
We again use homological normalizations (i.e. uniformizers correspond to arithmetic Frobenii and Hodge--Tate weights are normalized so that $\epsilon$ has Hodge--Tate weight $1$) for the Galois representations associated to modular forms, to be consistent with the literature we quote.

There is a newform $f$ of weight $2$ and level $\Gamma_1(13)$, with $q$-expansion (see \cite{LMFDB1324a})
	\begin{equation}\label{eq:qexp}
	f(q) = q + (-\zeta_6-1)q^2 + (2\zeta_6-2)q^3 + \zeta_6q^4 + (-2\zeta_6+1)q^5 + \cdots,
	\end{equation}
where $\zeta_6 = e^{\frac{2\pi i}{6}}$. 
The coefficient field of $f$ is $\Q(\sqrt{-3})$, and the nebentypus of $f$ is the character $(\Z/13\Z)^\times \rightarrow \Q(\sqrt{-3})^\times$ given by $2\mapsto \zeta_6$.
Further, $f$ and its conjugate are a basis for $S_2(\Gamma_1(13))$. 

We take $E = \Q_3(\sqrt{-3})$, and let
	\[ \rho_f : G_{\Q,\{3,13\}} \lra \GL_2(E) \]
denote the $(\sqrt{-3})$-adic representation attached to $f$, and let
	\[ \rhobar_f : G_{\Q,\{3,13\}} \lra \GL_2(\F_3) \]
be the residual representation.
The residual representation $\rhobar$ is absolutely irreducible, however its restriction to $G_{\Q(\zeta_3)}$ has abelian image. 
In particular, $H^0(G_3,\ad^0(\rhobar)(1)) \ne 0$. 

Let $\mu : G_{\Q,\{3,13\}} \rightarrow \calO^\times$ be the continuous character such that $\mu\epsilon^{-1}$ is the quadratic character of $\Gal(\Q(\sqrt{13})/\Q)$. 
Then $\det\rhobar_f = \mubar$, the reduction of $\mu$ modulo $\frakm_{\calO}$.
Let $R^\mu$ be the universal determinant $\mu$ deformation ring for $\rhobar$. 
Every irreducible component of $\Spec R^\mu$ has dimension at least $3$ by \cref{thm:dimdet}.
Let $R^{\loc} = R_3^{\square,\mu}$. 
By \cref{thm:gendetpres} and the local Euler--Poincar\'{e} characteristic formula, $\dim R^{\loc} \ge 7$. 
We fix a lift in the class of the universal determinant $\mu$ deformation of $\rhobar$, hence a $\CNL_{\calO}$-algebra morphism $R^{\loc} \rightarrow R^\mu$.
Note that $\rho_f|_{G_3}$ is crystalline with Hodge--Tate weights $\{0,1\}$, and $\det\rho_f|_{I_3} = \mu|_{I_3}$. 
So twisting $\rho_f|_{G_3}$ by an unramified character, if necessary, we see that $R_3^{(0,0),\cris,\mu} \ne 0$.
We let $X^{\loc} =\Spec R_3^{(0,0),\cris,\mu} \subset \Spec R^{\loc}$, so $\dim X^{\loc} = 5$ by \cref{thm:crdefringdet}. 
In particular, for any irreducible component $\calC$ of $\Spec R^\mu$, we have
	\[ \dim \calC + \dim X^{\loc} - \dim R^{\loc} \ge 1.\]
But, letting $X = X^{\loc} \times_{\Spec R^{\loc}} \Spec R^\mu$, we claim there is an irreducible component $\calC$ of $\Spec R^\mu$ such that 
	\[\calC \cap (X\otimes_{\calO} E) = \emptyset.\]
(However, the justification below does show that a different choice of $X^{\loc}$ yields a nontrivial intersection with  $\calC$.)

To see this, first note that by a version of Carayol's Lemma \cite{DiamondRefined}*{Lemma~2.1}, there is a modular lift $\rho_g$ of $\rhobar_f$ with $g\in S_2(\Gamma_1(9\cdot 13))$ and $\det\rho_g = \mu$. 
By local global compatibility, we have $\rho_g|_{I_{13}} = \mu|_{I_{13}} \oplus 1$. 
Let $\calC$ be an irreducible component of $\Spec R^\mu$ containing the point induced by $\rho_g$. 
Under the map $R_{13}^\square \rightarrow R^\mu$, the image of $\calC$ in $\Spec R_{13}^\square$ is contained in some characteristic $0$ irreducible component of $\Spec R_{13}^\square$. 
Inertial types are locally constant on $\Spec R_{13}^\square[1/p]$ (see \cite{GeeTypes}*{\S2}), so we deduce that for any $\Qbar_p$-point $x$ on $\calC$, the semisimplification of $\rho_x|_{I_{13}}$ is isomorphic to $\mu|_{I_{13}}\oplus 1$. 
Since $\mu|_{I_{13}}\ne 1$, we further conclude that $\rho_x|_{I_{13}} \cong\mu|_{I_{13}}\oplus 1$.

Now assume that $\calC \cap (X\otimes_{\calO} E) \ne 0$, and choose a $\Qbar_p$-point $x$ in this intersection. 
Then $\rho_x$ is unramfied outside of $\{3,13,\infty\}$, is crystalline at $3$ with Hodge--Tate weights $\{0,1\}$, and $\rho_x|_{I_{13}} \cong\mu|_{I_{13}}\oplus 1$. 
The $q$-expansion \eqref{eq:qexp} shows $f$ is $3$-ordinary, so
	\[ \rhobar_f|_{G_3} \cong 
	\begin{pmatrix}
	\chi_1 & \ast \\ & \chi_2
	\end{pmatrix},
	\]
with $\chi_2$ unramified and $\chi_1|_{I_3} = \epsilonbar|_{I_3}$. 
Using \cite{KW2}*{Lemma~3.5}, we see that $\rho_x$ is also ordinary, and since $\rhobar_f$ is $3$-distinguished, we can apply \cite{SWirreducible}*{Theorem in \S1} to conclude that $\rho_x$ is modular. 
But by local global compatibility, it must arrise from an eigenform in $S_2(\Gamma_1(13))$ with a quadratic nebentypus. 
This is a contradiction as $S_2(\Gamma_1(13))$ is spanned by $f$ and its conjugate, and the nebentypus of $f$ has order $6$.

\end{document}